\documentclass{amsart}
\usepackage{geometry}
\usepackage{amsthm,amssymb,amsmath,hyperref}
\usepackage{graphicx}
\usepackage{color}
\usepackage{tikz}
\usetikzlibrary{arrows.meta}
\usepackage[utf8]{inputenc}
\usepackage[T1]{fontenc}
\allowdisplaybreaks

\theoremstyle{plain}
\newtheorem{thm}{Theorem}[section]

\newtheorem{cor}[thm]{Corollary}
\newtheorem{lem}[thm]{Lemma}
\newtheorem{prop}[thm]{Proposition}

\theoremstyle{definition}
\newtheorem{defn}[thm]{Definition}

\theoremstyle{remark}
\newtheorem{rem}[thm]{Remark}

\theoremstyle{plain}

\numberwithin{equation}{section}

\newcommand{\B}{{\mathbb B}}

\newcommand{\R}{{\mathbb R}}

\newcommand{\Z}{{\mathbb Z}}

\newcommand{\calF}{{\mathcal F}}

\newcommand{\calS}{{\mathcal S}}
\newcommand{\calT}{{\mathcal T}}

\def\udot#1{\ifmmode\oalign{$#1$\crcr\hidewidth.\hidewidth
    }\else\oalign{#1\crcr\hidewidth.\hidewidth}\fi}

\def\R{\mathbb{R}}
\def\Z{\mathbb{Z}}
\def\T{\mathbb{T}}

\renewcommand{\parallel}{\mathrel{/\mkern-5mu/}}
\makeatletter
\newcommand{\notparallel}{%
  \mathrel{\mathpalette\not@parallel\relax}%
}
\newcommand{\not@parallel}[2]{%
  \ooalign{\reflectbox{$\m@th#1\smallsetminus$}\cr\hfil$\m@th#1\parallel$\cr}%
}

\begin{document}

\title[Unstable Cahn-Hilliard in 2D with shear]{Global Existence for the unstable Cahn-Hilliard equation in 2D with a Shear Flow}
\author{Bingyang Hu, Dinghua Xu, Yeyu Zhang}

\address{Bingyang Hu: Department of Mathematics and Statistics, Auburn University, 211 Parker Hall, Auburn, AL 36849, U.S.A.}%
\email{bzh0108@auburn.edu}

\address{Dinghua Xu: School of Mathematics, Shanghai University of Finance and Economics, Shanghai, 200433, PR China; 
School of Science, Zhejiang Sci-Tech University, Hangzhou,  310023, PR China}%
\email{dhxu6708@mail.shufe.edu.cn; dhxu6708@zstu.edu.cn}

\address{Yeyu Zhang: School of Mathematics, Shanghai University of Finance and Economics, Shanghai, 200433, PR China}%
\email{zhangyeyu@mail.shufe.edu.cn}

\begin{abstract}

In this paper, we study the advective unstable Cahn--Hilliard equation on $\T^2$ with shear flow:
\begin{equation*}
    \begin{cases}
u_t+Av_1(y) \partial_x u+\varepsilon \Delta^2 u= \Delta(a u^3+ b u^2)  \quad & \quad  \textrm{on} \quad \T^2;  \\
\\
u \ \textrm{periodic} \quad & \quad \textrm{on} \quad  \partial \T^2,
\end{cases}
\end{equation*}
where $u_0\in H_0^2(\T^2)$, $A,\varepsilon>0$, $a<0$, and $b\in\R$. The condition $a<0$ puts the model in an unstable phase-field regime: the nonlinear chemical potential may amplify, rather than restore, concentration fluctuations, as in spinodal decomposition. The shear term $Av_1(y)\partial_xu$ models imposed stirring along the shear direction; through mixing, it enhances dissipation and counteracts the growth driven by the unstable cubic term $\Delta(au^3)$. Assuming that the shear profile has finitely many critical points and that linearly growing modes occur only in the shear direction, we prove that the $L^2$-energy converges exponentially to zero, provided $|a|$ and $\|\int_\T u_0(x,\cdot)\,dx\|_{L_y^2}$ are sufficiently small.

\end{abstract}
\date{\today}

\thanks{}

\maketitle
\tableofcontents

\section{Introduction}

In this paper, we are interested in the following initial boundary value problem of the non-linear Cahn-Hilliard equation with advection:
\begin{equation} \label{maineq}
\begin{cases}
u_t+{\bf v} \cdot \nabla u+\varepsilon \Delta^2 u=\Delta(a u^3+bu^2+c u)  \quad & \quad  \textrm{on} \quad \T^2;  \\
\\
u \ \textrm{periodic} \quad & \quad \textrm{on} \quad  \partial \T^2, 
\end{cases}
\end{equation}
with initial data $u(x, 0)=u_0(x) \in H^2(\T^2)$. Here  $\varepsilon>0$, $a, b, c \in \R$ and ${\bf v}$ is an incompressible flow on $\T^2$. In the case when ${\bf v}=0$, \eqref{maineq} becomes the classical Cahn-Hilliard equations, which arises in the study of phase separation in cooling binary solutions such as alloys, glasses and polyer mixtures (see, e.g.,  \cite{CH58, Cahn61}). More broadly, phase-change phenomena coupled with transport and wave effects also appear in other fluid-dynamical settings, including moist atmospheric dynamic (see, e.g.,  \cite{ZSS21a,ZSS21b,ZSS22}). It is well-known that if $a>0$, then the solution to the problem \eqref{maineq} without advection spontaneously forms domains separated by thin transition regions (see, e.g., \cite{ES86, Ell89, Peg89}). In the phase-field interpretation, this corresponds to a stable bulk free-energy landscape with preferred phases, so that the dynamics tends to separate the mixture into nearly pure phases. By contrast, when $a<0$, the nonlinear chemical potential has a destabilizing sign (see, e.g., \cite{Cahn65, LeeKimJeongShin14}). In the phase-field interpretation, this corresponds to an unstable free-energy landscape in which concentration fluctuations may be amplified rather than restored, as in the classical spinodal-decomposition mechanism. This places the equation in an unstable regime, in which the solution in general blows up (see, e.g., \cite{ES86, GV02}). Therefore, a natural question that one can ask is the following: \emph{can either the phase separation or the blow-up be suppressed if we add an extra advection term?}

\vspace{0.1cm}

In \cite{FFIT19}, the authors studied the global existence for the problem \eqref{maineq} with $a=1, b=0$ and $c=-1$ (which is a typical case for the \emph{stable} case) under the assumption that the  \emph{dissipation time} of ${\bf v}$ is sufficiently small. As a conclusion, they showed that the solution to \eqref{maineq} converges exponentially to the total concentration 
\begin{equation} \label{mean0cond}
M:=\int_{\T^2} u_0 dxdy=\int_{\T^2} udxdy
\end{equation}
(note that \eqref{maineq} conserves mean of the solution), and in particular, they showed that  even in the situation when $\varepsilon>0$ is sufficiently small, no phase separation occurs if the stirring velocity field is sufficiently mixing. Here the dissipation time of ${\bf v}$ is defined as follows.

\begin{defn}
Consider the hyper-diffusion equation on $\T^2 \times (0, \infty)$. 
\begin{equation} \label{hyper-diffusion}
\theta_t+{\bf v} \cdot \nabla \theta+\varepsilon \Delta^2 \theta=0.
\end{equation}
Let $S_{s, t}$ be the solution operator to \eqref{hyper-diffusion}, that is for any $f \in L^2(\T^2)$, the function $\theta(t)=S_{s, t} f$ solves \eqref{hyper-diffusion} with initial data $\theta(s)=f$, and periodic boundary conditions. The \emph{dissipation time} of ${\bf v}$ is
$$
\tau_{\textrm{dis}}:=\inf \left\{ t \ge 0 \bigg | \left\|S_{s, s+t}\right\|_{L^2_0 \to L^2_0} \le \frac{1}{2} \quad \textrm{for all} \quad s \ge 0 \right\},  
$$
where $L_0^2(\T^2)$ is the collection of all square integrable functions on $\T^2$ with mean zero. 
\end{defn}

In this paper, we are interested in whether one can apply a similar idea in \cite{FFIT19} to study the \textit{unstable} case, in which, the constant $a$ in \eqref{maineq} is strictly negative, and it turns out the answer is \emph{affirmative} if in addition, we assume $|a|$ is sufficiently small. Note that the major difference in this case is that the double well potential no longer has a global lower bound, and hence one has to handle the term $\Delta(u^3)$ properly. 

Now we turn to some details. Let us consider the problem \eqref{maineq} with $a<0, b \in \R, c=0$ and when $v$ is given by a \emph{shear flow}, which is of the form
$$
{\bf v}=\begin{pmatrix}
Av_1(y) \\
0 
\end{pmatrix},
$$
where $v_1 \in W^{1, \infty}(\T)$ and $A>0$ is its amplitude. Namely, we consider the problem
\begin{equation} \label{maineq01}
\begin{cases}
u_t+Av_1(y) \partial_x u+\varepsilon \Delta^2 u= \Delta(a u^3+ b u^2)  \quad & \quad  \textrm{on} \quad \T^2;  \\
\\
u \ \textrm{periodic} \quad & \quad \textrm{on} \quad  \partial \T^2, 
\end{cases}
\end{equation}
with $a<0, b \in \R$ and initial data $u(x, 0)=u_0(x) \in H_0^2(\T^2)$, which is the collection of all $H^2(\T^2)$-functions with mean zero. Our goal is to show the global existence of solutions to the problem \eqref{maineq02} under certain proper assumption on $v$. The study of suppression of blow-up in nonlinear parabolic equations is a recent
research topic, and we refer the interested reader to \cite{BH17, CDFM21, FHXZ21, He18} and the references therein
for more detail.

We first note that by a rescaling argument, we can rewrite \eqref{maineq01} as 
\begin{equation} \label{maineq02}
\begin{cases}
u_t+v_1(y) \partial_x u+\varepsilon \gamma \Delta^2 u= \gamma \Delta(a u^3+ b u^2)  \quad & \quad  \textrm{on} \quad \T^2;  \\
\\
u \ \textrm{periodic} \quad & \quad \textrm{on} \quad  \partial \T^2, 
\end{cases}
\end{equation}
where $\gamma=A^{-1}$. Next, we turn to the assumption on the flow ${\bf v}$, which is now ${\bf v}=\begin{pmatrix}
v_1(y) \\
0
\end{pmatrix}$ after rescaling. The key assumption here is that the linear operator
\begin{equation} \label{20210822eq05}
H_\gamma=H_{\varepsilon, \gamma}:=\varepsilon \gamma \Delta^2+v_1(y) \partial_x
\end{equation} 
is \emph{dissipation enhancing} (see, e,g., \cite{CKRZ08, CDE20, FI19}) on the orthogonal component to the kernel of the transport operator $v_1(y)\partial_x$. More precisely, for any $g \in L^2(\T^2)$, we decompose it as 
$$
\langle g \rangle(t, y):=\int_\T g(t, x, y)dx \quad \textrm{and} \quad g_{\notparallel}(t, x, y)=g(t, x, y)-\langle g \rangle (t, y),
$$
where $\langle g \rangle$ is the \emph{average component} and $g_{\notparallel}$ is the \emph{fluctuation component}. It is easy to see that 
$$
\langle g \rangle \in \textrm{Ker} \left(v_1(y) \partial_x \right) \quad \textrm{and} \quad g_{\notparallel} \in \left(  \textrm{Ker} \left(v_1(y) \partial_x \right) \right)^{\perp}
$$

\begin{defn} \label{defn01}
We say the shear flow ${\bf v} =\begin{pmatrix} 
v_1(y) \\
0
\end{pmatrix}$ is a \emph{horizontal polynomial mixing shear flow} if there exists some global constant $C_1>0$ and $m \ge 2$, such that
\begin{equation} \label{20210822eq01}
\left\| e^{-(v_1 \partial_x)t} g_{\notparallel} \right\|_{H^{-1}} \le \frac{C_1}{(1+t)^m} \|g_{\notparallel}\|_{H^1}, \quad t \ge 0,
\end{equation} 
for any $g \in L_0^2(\T^2)$. 
\end{defn}

\begin{rem}
The constant $m$ in Definition \ref{defn01} is closely related to the flow function $v_1$. For example, it has been shown in \cite{BC17} that if $u$ has a finite number of critical points of order at most $m \ge 2$, namely at most $m-1$ derivatives vanish at the critical points, then there exists some constant $C_1>0$, such that the estimate \eqref{20210822eq01} holds. 
\end{rem}

As a consequence of \eqref{20210822eq01}, we have the following result. 

\begin{prop}{\cite[Theorem 2.1]{CDE20}} \label{20210823prop01}
For each $\gamma>0$ and ${\bf v}=\begin{pmatrix}
v_1(y) \\
0
\end{pmatrix}$ be a horizontal polynomial mixing shear flow. Then for any $g \in L^2(\T^2)$ with mean zero, one has for any $t \ge 0$,
\begin{equation} \label{20210822eq02}
\left\|e^{-t H_\gamma} g_{\notparallel}\right\|_{L^2} \le 10 e^{-\lambda_\gamma  t} \|g_{\notparallel}\|_{L^2}, \quad \lambda_\gamma=\lambda_{\varepsilon, \gamma}:=C_{2, \varepsilon} \gamma ^{\frac{2}{2+m}},
\end{equation}
where $C_{2, \varepsilon}>0$ is an absolute constant independent of the choice of $\gamma$, and only depending on $\varepsilon$, $C_1$, $m$  (which are defined in \eqref{20210822eq01}) and any dimensional constants. 
\end{prop}

We are ready to state the main theorem of this paper. 

\begin{thm} \label{mainthm}
Let $u_0 \in H_0^2(\T^2)$, $|a|$ defined in \eqref{maineq02} be sufficiently small and ${\bf v}=\begin{pmatrix}
v_1(y) \\
0
\end{pmatrix}$ be a horizontal polynomial mixing shear flow with parameters $C_1>0$ and $m \ge 2$. If
$$
\left\|\int_{\T} u_0(x, \cdot) dx \right\|_{L_y^2} \ll 1, 
$$
then there exists some $\gamma^*>0$,  which only depends on $\varepsilon, a, b, c$ and $\left\|u_0\right\|_{L^2}$, such that for any $0< \gamma \le \gamma^*$, there exists a global-in-time weak solution (or mild solution) of \eqref{maineq02} with initial data $u_0$, such that $u \in L^\infty \left([0, \infty), L^2(\T^2) \right) \cap L^2 \left([0, \infty), H^2(\T^2) \right)$. Moreover, $\|u(t)\|_{L^2}$ converges to $0$ exponentially. 
\end{thm}

\begin{rem} \label{20211026rem01}
\begin{enumerate}
    \item [(1).] In the sequel, we will prove Theorem \ref{mainthm} in a quantitative way, in particular, we will be precise on the size of $|a|$ and $\left\| \int_{\T} u_0(x, \cdot) dx \right\|_{L_y^2}$; 

\item [(2).] Let us comment on the roles of the parameter $|a|$ and the flow amplitude $A$. Recall that, after the rescaling leading to \eqref{maineq02}, one has $\gamma=A^{-1}$. Moreover, by Proposition \ref{20210823prop01}, the enhanced-dissipation rate associated with the shear flow is
$$
\lambda_\gamma=C_{2,\varepsilon}\gamma^{\frac{2}{2+m}}.
$$
Thus, on the enhanced-dissipation time scale, the Duhamel terms associated with the nonlinearity (see \eqref{20210823eq01} and \eqref{20210905eq01}) naturally involve the factor
$$
\frac{\gamma}{\lambda_\gamma}
\sim \gamma^{\frac{m}{m+2}}
=A^{-\frac{m}{m+2}}.
$$
This scaling shows that increasing the flow amplitude $A$ strengthens the stabilizing effect of the shear flow by reducing the accumulation of nonlinear terms over the enhanced-dissipation time scale.

However, increasing the flow amplitude $A$ through the above scaling does \emph{not} by itself remove the difficulty caused by a large destabilizing cubic coefficient $|a|$. Indeed, the term $\Delta(a u^3)$ enters the nonlinear estimates through the size of $|a|$, and its contribution has to be controlled within the bootstrap argument. Thus, in Theorem \ref{mainthm}, the smallness of $|a|$ should be understood not merely as a consequence of the rescaling, but as part of a sufficient quantitative regime in which the cubic destabilizing effect can be dominated by the enhanced dissipation induced by the shear flow. This balance is reflected in Propositions \ref{bootestprop1} and \ref{bootestprop3}. 

In view of the above discussion, it is natural to ask whether sufficiently strong shear flows can still suppress blow-up for large $|a|$, and what the corresponding sharp relation between $|a|$ and the flow amplitude $A$ should be.

\vspace{0.1cm}

\item [(3).] Theorem \ref{mainthm} shows that the $L^2$ blow-up of the unstable Cahn--Hilliard equation can be suppressed via a shear flow. More precisely, let us consider the unstable Cahn--Hilliard equation without advection, that is
    \begin{equation} \label{unstable}
    \begin{cases}
    u_t+\varepsilon \Delta^2 u= \Delta(a u^3+ b u^2)  \quad & \quad  \textrm{on} \quad \T^2;  \\
\\
u \ \textrm{periodic} \quad & \quad \textrm{on} \quad  \partial \T^2, 
    \end{cases}
    \end{equation} 
where $\varepsilon>0$, $a<0$ and $b \in \R$. In the phase-field interpretation, the regime $a<0$ corresponds to an unstable free-energy landscape in which concentration fluctuations may be amplified rather than restored. It is well known that if the \emph{Landau--Ginzburg free energy} of the initial data is sufficiently negative, that is, if
$$
-\int_{\T^2} \left(H(u_0)+\frac{\gamma}{2} |\nabla u_0|^2 \right) dxdy
$$
is sufficiently large, where
\begin{equation} \label{20211027eq21}
H(u)=\int_0^u \left(as^3+bs^2 \right) ds=\frac{au^4}{4}+\frac{bu^3}{3}
\end{equation} 
is the bulk potential, then there exists a $T^*>0$, such that
\begin{equation} \label{20211025eq01}
\limsup_{t \to T^*_{-}} \|u(t)\|_{L^2}=\infty. 
\end{equation} 
(see, e.g., \cite[Theorem 3.1]{ES86}). Note that under the assumptions of Theorem \ref{mainthm}, it is possible for the Landau--Ginzburg free energy of the initial data to be very negative; for example, one can take an appropriate $b$ with $|b|$ sufficiently large. Hence the solution of the corresponding non-advective problem \eqref{unstable} may blow up in finite time, while our result shows that, with an additional shear flow satisfying \eqref{20210822eq01}, the solution can exist globally.

\end{enumerate}
\end{rem}

\vspace{0.1cm}

\noindent {\bf Novelty and comparison with previous works.} We now clarify the main differences between the present work and the closely related papers \cite{FFIT19, FHXZ21}. Although our argument is inspired by the general philosophy of using mixing or enhanced dissipation to suppress instability, the unstable Cahn--Hilliard equation studied here contains difficulties that are absent in the stable case.

In \cite{FFIT19}, the authors considered the stable advective Cahn--Hilliard equation, where the effect of the imposed velocity field is quantified by the dissipation time of the associated advection-hyperdiffusion equation on the whole mean-zero space. Once this dissipation time is sufficiently small, the mixing mechanism acts on the full fluctuation around the conserved spatial average and drives the solution toward the homogeneous mixed state. The shear-flow setting in the present paper is structurally different: the operator
\[
H_{\gamma}=\varepsilon\gamma\Delta^2+v_1(y)\partial_x
\]
is dissipation enhancing only on the orthogonal complement of the kernel of \(v_1(y)\partial_x\). Thus the solution must be decomposed into the shear-invariant part \(\langle u\rangle\) and the enhanced-dissipative part \(u_{\notparallel}\), and \(\langle u\rangle\) has to be controlled by separate nonlinear estimates rather than directly by the mixing effect. In this sense, although our result assumes \( |a| \) to be small, the main theorem is not a direct perturbative consequence of the stable result in \cite{FFIT19}.

Another essential difference comes from the sign of the cubic term. In the stable equation in \cite{FFIT19}, the cubic term is compatible with the dissipative energy structure. In the present unstable regime \(a<0\), however, the potential \(H(u)\) defined in \eqref{20211027eq21} is not bounded from below, and the contribution of \(\Delta(a u^3)\) may drive growth rather than dissipation. Hence a global a priori energy bound is unavailable. Our proof replaces this missing energy control by a bootstrap argument: enhanced dissipation first yields decay and integrated \(H^2\)-type control for \(u_{\notparallel}\); this is used to control \(\langle u\rangle\); and the resulting bound is then fed back into the estimates for \(u_{\notparallel}\), closing the bootstrap; we refer the reader to Figure \ref{fig:bootstrap-roadmap} for an outline of how the bootstrap argument is organized.
\vspace{0.1cm}

The present work also differs from the earlier work \cite{FHXZ21} by Feng, Xu, and the first and third authors of the present paper. The equation studied there is a second-order non-local semilinear parabolic equation, whose nonlinear source term contains no spatial derivatives. In the present fourth-order Cahn--Hilliard problem, however, the nonlinearity takes the form
\[
\Delta(a u^3+b u^2).
\]
After expansion, this produces derivative nonlinearities such as
\[
u^2\Delta u, \qquad u|\nabla u|^2,
\]
which are considerably more delicate than the semilinear source terms treated in \cite{FHXZ21}. Moreover, after decomposing
\[
u=\langle u\rangle+u_{\notparallel},
\]
these derivative nonlinearities generate coupled terms between the zero-mean component $u_{\notparallel}$ and the shear-invariant component $\langle u\rangle$. Since the shear flow only directly enhances dissipation on the zero-mean component, while $\langle u\rangle$ lies in the kernel of the transport operator, the two components must be controlled simultaneously. This coupled fourth-order bootstrap structure is the main technical novelty of the present paper. 

\medskip 

The structure of this paper is as follows. Section 2 deals with the local existence to the problem \eqref{maineq}. In Section \ref{bootassumption}, we establish several local estimates for the terms $\left\|u_{\notparallel}(t)\right\|_{L^2}$ and $\gamma \int_s^t \left\|\Delta u_{\notparallel} \right\|_{L^2}^2 d\tau$, and this allows us to make the bootstrap assumptions to the problem \eqref{maineq02}. As a consequence of the bootstrap assumptions, in Section 4, we prove uniform bounds for the term $\langle u \rangle$ with both $|a|$ and $\gamma$ being sufficiently small. Section 5 is devoted to prove the main theorem \ref{mainthm}. We prove it via a bootstrap argument by showing the bootstrap assumptions can be improved.

Finally, in this paper, we will write $B$ and $C$ as some constants that might change line by line, where
\begin{enumerate}
\item [(1).] $B$ will only depend on $\varepsilon, b$ and any dimensional constants;
\item [(2).] $C$ will only depend on $\varepsilon, a, b, \|\langle u \rangle (0) \|_{L^2}$,  $\left\|u_{\notparallel}(0)\right\|_{L^2}$ and any dimensional constants. 
\end{enumerate} 
{\bf Acknowledgements.} 
B. Hu acknowledges support from the Simons Travel Grant MPS-TSM-00007213.
D. Xu acknowledges support from the National Natural Science Foundation of China (NSFC) under Grants No. 12371428 and No. 11871435.
Y. Zhang acknowledges support from the National Natural Science Foundation of China (NSFC) under Grants No. 12401562, No. 12571459, and No. 12241103. The authors would also like to thank the anonymous referee for the careful reading and valuable suggestions, which led to several improvements in the presentation of the paper.

\section{Preliminary: Local existence}  \label{localext} 

In this section, we study the local existence of the solutions to the problem \eqref{maineq} with arbitrary $H^2$ initial data. Here, instead of proving a prior estimate (see, \cite{ES86}), it is more convenient for us to consider the \emph{mild solutions} of \eqref{maineq}, which will play an important role in our later context when we consider the case when $v$ is a shear flow. 

The result in this section is standard. Nevertheless, we would like to make a remark that it is already known that there exists a global solution to the equation \eqref{maineq} if $a>0$ (see, e.g., \cite[Theorem 1.1]{ES86} and \cite[Proposition 2.1]{FFIT19}), while the global existence is not guaranteed if $a<0$ (see, \cite{ES86}). To this end, we make a remark that the constant $C$ used in this section is also allowed to depend on $c$, as we do not require $c=0$ in the current section. 

We start with some basic setup. Given a function $f \in L^1(\T^2)$, we denote $\hat{f}({\bf k})$ to be its Fourier coefficient of $f$ at frequency ${\bf k} \in \Z^2$, and hence $\hat{f}:= \left\{\hat{f}(\bf k)\right\}_{\bf k \in \Z^2}$. Then the \emph{inhomogeneous Sobolev space} $H^s(\T^2), s \in \R$ is defined to be the collection of all measure functions $f$ on $\T^2$, with
$$
\|f\|_{H^s}^2:=\sum_{{\bf k} \in \Z^2} (1+|{\bf k}|^2)^s \left| \hat{f}({\bf k}) \right|^2=\left\| \left(I-\Delta \right)^{s/2} f \right\|_2^2<\infty, 
$$
while the \emph{homogeneous Sobolev space} $\dot{H}^s(\T^2), s \in \R$ consists of all measurable functions $f$ on $\T^2$ with 
$$
\|f\|_{\dot{H}^s}^2:=\sum_{{\bf k} \in \Z^2} |{\bf k}|^{2s} \left| \hat{f}({\bf k}) \right|^2=\left\| \left(-\Delta \right)^{s/2} f \right\|_2^2<\infty.
$$
Note that it is clear that for $f \in L^2(\T^2)$, $f \in H^s(\T^2)$ if and only if $f \in \dot{H}^s(\T^2)$.

Let $e^{-t\Delta^2}$ be the semigroup generated by the bi-Laplacian $\Delta$, namely, 
$$
e^{-t\Delta^2}f:= \calF^{-1} \left( \left\{e^{-t\left|{\bf k} \right|^4} \hat{f}(\bf k) \right\}_{{\bf k} \in \Z^2} \right),
$$
where $\calF^{-1}$ is the inverse Fourier transform on $\Z^2$: for any $\left\{a_{\bf k} \right\}_{{\bf k} \in \Z^2}$, 
$$
\calF^{-1} \left( \left\{a_{\bf k} \right\} \right)_{\bf k \in \Z^2}(x):=\sum_{\bf k \in \Z^n} a_{\bf k} e^{-2\pi x \cdot \bf k}, \quad x \in \T^2. 
$$

The following semi-group estimate for $e^{-t \Delta^2}$ is standard and we would like to leave the proof to the interested reader. 

\begin{lem} \label{20210728lem01} 
For any $s>0$, there exists a dimension constant $C>0$ such that
$$
\left\| \left(-\Delta \right)^{\frac{s}{2}}e^{-t\Delta^2} f \right\|_{L^2} \le Ct^{-\frac{s}{4}} \|f\|_{L^2}. 
$$
\end{lem}

\medskip

We are ready to introduce the definition of the mild solutions to the problem \eqref{maineq}. 

\begin{defn} \label{mildsol}
Let $v \in L^\infty([0, \infty); W^{1, \infty})$ be a divergence free flow.  A function $u \in C \left([0, T]; H^2\right) \cap L^2\left( \left(0, T \right]; H^4 \right)$, $T>0$, is called a \emph{mild solution} of \eqref{maineq} on $[0, T]$ with initial data $u_0 \in H^2$, if for any $0 \le t \le T$, 
\begin{eqnarray} \label{mildsoleq01}
&& u(t)=\calT(u)(t):=e^{-t\varepsilon \Delta^2}u_0+\int_0^t e^{-(t-s)\varepsilon \Delta^2} \Delta \left(a u^3+b u^2+cu \right) ds \nonumber \\
&& \quad \quad \quad \quad  \quad \quad \quad \quad \quad   -\int_0^t e^{-(t-s)\varepsilon \Delta^2}\left(v \cdot \nabla u \right)ds
\end{eqnarray}
holds pointwisely in time with values in $H^2$, where the integral is defined in B\"ochner sense. 
\end{defn} 

\begin{rem}
We claim that the second integral in \eqref{mildsoleq01} is well-defined. Indeed, if $u \in C \left([0, T]; H^2 \right)$, we have
\begin{eqnarray} \label{20210727eq01} 
&&\left\|\Delta(au^3+ bu^2+cu) \right\|_{L^2} \nonumber \\
&& \quad \quad \quad = \left\|3a u^2 \Delta u+6a u|\nabla u|^2+2b|\nabla u|^2+2bu \Delta u+|c|\Delta u \right\|_{L^2} \nonumber\\
&& \quad \quad \quad \le 3|a| \|u\|_{L^\infty}^2 \|\Delta u \|_{L^2}+6|b| \|u\|_{L^\infty} \|\nabla u \|_{L^4}^2+c\|\Delta u \|_{L^2} \nonumber \\
&& \quad \quad \quad \quad \quad + 2 |b| \|\nabla u\|_{L^4}^2+2 |b| \|u\|_{L^\infty} \|\Delta u\|_{L^2}.
\end{eqnarray}
By the Gagliardo-Nirenberg's inequalities with $n=2$, 
\begin{equation} \label{20210729eq20}
\left\|u \right\|_{L^\infty} \le C \left(\left\|\Delta u \right\|_{L^2}^{\frac{1}{2}} \left\|u \right\|_{L^2}^{\frac{1}{2}}+\|u\|_{L^2} \right) \le C \|u\|_{H^2}, 
\end{equation} 
\begin{equation} \label{20210729eq21}
\left\|\nabla u \right\|_{L^4} \le C \left(\left\|\Delta u \right\|_{L^2}^{\frac{3}{4}} \left\|u \right\|_{L^2}^{\frac{1}{4}}+\|u\|_{L^2} \right) \le C\|u\|_{H^2}
\end{equation} 
we can bound the right hand side of \eqref{20210727eq01} by 
\begin{equation} \label{20210727eq02} 
C_{a, b, c} \left( \|u\|_{H_2}^3+\|u\|_{H_2}^2+\|u\|_{H^2} \right), 
\end{equation} 
where $C_{a, b, c}$ is some positive constant which only depends on $a, b$ and $c$. The desired claim then follows clearly. 

\end{rem}

We shall use the Banach contraction mapping to construct such a mild solution. For this purpose, we introduce the following Banach space. For any $T>0$, we define 
$$
X_T:=C \left([0, T]; H^2 \right) \cap \left\{u: \T^2 \times \R_+ \to \R \bigg | \sup_{0<t \le T} \left(t^{\frac{1}{4}} \| \Delta^{\frac{3}{2}}  u \|_{L^2}+t^{\frac{1}{2}} \|\Delta^2 u\|_{L^2} \right)<\infty \right\}
$$
with the norm
$$
\|u\|_{X_T}:=\max \left( \sup_{0 \le t\le T} \|u\|_{H^2}, \sup_{0 \le t \le T} \left(t^{\frac{1}{4}} \|\Delta^{\frac{3}{2}} u \|_{L^2}+t^{\frac{1}{2}} \|\Delta^2 u \|_{L^2}\right)  \right).
$$
We have the following result. 

\begin{thm} \label{20210728thm01} 
Let $v \in L^\infty([0, \infty); W^{1, \infty})$ be a divergence free flow. Then there exists $0<T \le 1$ depending on $\varepsilon$, $a$, $b$, $c$, $\sup_{t \ge 0} \|v\|_{W^{1, \infty}}$ and $\|u_0\|_{H^2}$ such that \eqref{maineq} admits a unique mild solution $u$ on $[0, T]$, which is unique in $X_T$. 
\end{thm}

The proof of Theorem \ref{20210728thm01}  consists of several lemmas. 

\begin{lem} \label{20210728lem02} 
Under the assumption of Theorem \ref{20210728thm01}, we have for any $0<T \le 1$, 
$$
\calT(u) \in C([0, T]; H^2).
$$
Moreover, for each $t \in (0, T]$, there exists some $C=C(\varepsilon, a, b, c)>0$, such that 
\begin{eqnarray} \label{20200728eq10}
&& \left\|\calT(u)(t) \right\|_{H^2} \nonumber \\
&& \quad \quad \le C \left( \|u_0\|_{H^2}+ t^{\frac{1}{2}} \left(  \|u\|_{X_T}^3+ \|u\|_{X_T}^2+ (1+\|v\|_{L^\infty([0, \infty); W^{1, \infty})} )  \|u\|_{X_T} \right) \right).  
\end{eqnarray} 
\end{lem}

\begin{proof}
It suffices to prove the estimate \eqref{20200728eq10}. For any $t>0$, by Lemma \ref{20210728lem01}, we have
\begin{eqnarray*}
&&\left\|\Delta \calT(u)(t) \right\|_{L^2} \le \left\|e^{-t \varepsilon \Delta^2} \left(\Delta u_0 \right) \right\|_{L^2}+ \int_0^t \left\|\Delta e^{-t\varepsilon \Delta^2} (v \cdot \nabla u) \right\|_{L^2} ds\\
&& \quad \quad \quad \quad \quad \quad\quad \quad \quad\quad \quad \quad\quad  + \int_0^t \left\|\Delta e^{-(t-s)\varepsilon \Delta^2} \Delta \left(a u^3+b u^2+cu \right) \right\|_{L^2} ds  \\
&& \quad \le  C \left(\|u_0\|_{H^2}+\int_0^t (t-s)^{-\frac{1}{2}} \left[\left\|\Delta \left(a u^3+b u^2+cu \right) \right\|_{L^2}+\left\|v \cdot \nabla u \right\|_{L^2} \right] ds \right)  \\
&& \quad \le C \left(\|u_0\|_{H^2}+t^{\frac{1}{2}} \left(\|u\|_{X_T}^3+\|u\|_{X_T}^2+ (1+\|v\|_{L^\infty([0, \infty); L^\infty)}) \|u\|_{X_T} \right) \right)
\end{eqnarray*}
where we have used \eqref{20210727eq02}  in the above estimate. Similarly, we have
\begin{eqnarray*}
&& \|\calT(u)(t)\|_{L^2} \\
&& \quad \quad \le C \left(\|u_0\|_{H^2}+t \left( \|u\|_{X_T}^3+ \|u\|_{X_T}^2+ (1+\|v\|_{L^\infty([0, \infty); L^\infty)})  \|u\|_{X_T} \right) \right)
\end{eqnarray*}
and
\begin{eqnarray*}
&& \|\nabla \calT(u)(t)\|_{L^2} \\
&& \quad \quad \le C \left(\|u_0\|_{H^2}+t^{\frac{3}{4}} \left( \|u\|_{X_T}^3+ \|u\|_{X_T}^2+ (1+\|v\|_{L^\infty([0, \infty); L^\infty)})  \|u\|_{X_T} \right) \right). 
\end{eqnarray*}
Combining all these estimates yields \eqref{20200728eq10}. 
\end{proof}

\begin{lem} \label{20210728lem03} 
Under the assumption of Theorem \ref{20210728thm01}, we have for any $0<T \le 1$, there exists some $C=C(\varepsilon, a, b, c)>0$, such that 
\begin{eqnarray} \label{20210729eq01}
&& \sup_{0 \le t \le T} \left( t^{\frac{1}{4}}\left\|\Delta^{\frac{3}{2}} u \right\|_{L^2}+t^{\frac{1}{2}} \left\|\Delta^2 u \right\|_{L^2} \right) \nonumber \\
&&\quad \quad  \le C \left( \|u_0\|_{H^2}+ T^{\frac{1}{2}} \left( \|u\|_{X_T}^3+ \|u\|_{X_T}^2+ \left(1+\|v\|_{L^\infty([0, \infty); W^{1, \infty})}\right) \|u\|_{X_T} \right) \right). 
\end{eqnarray} 

\end{lem}

\begin{proof}
We first prove 
\begin{eqnarray} \label{20210729eq11}
 && \sup_{0 \le t \le T}  t^{\frac{1}{4}}\left\|\Delta^{\frac{3}{2}} u \right\|_{L^2}  \nonumber  \\
 && \quad \quad \le C \left( \|u_0\|_{H^2}+ t^{\frac{1}{2}} \left( \|u\|_{X_T}^3+ \|u\|_{X_T}^2+  \left(1+\|v\|_{L^\infty([0, \infty); W^{1, \infty})}\right) \|u\|_{X_T} \right) \right). 
\end{eqnarray} 
Indeed, by Lemma \ref{20210728lem01} and \eqref{20210727eq02} again, we have
\begin{eqnarray*}
&&\left\| \Delta^{\frac{3}{2}} \calT(u)(t) \right\|_{L^2} \le  \left\|\Delta^{\frac{1}{2}} e^{-t\varepsilon \Delta^2} \left(\Delta u_0 \right) \right\|_{L^2}+ \int_0^t \left\|\Delta^{\frac{3}{2}} e^{-(t-s)\varepsilon \Delta^2} (v \cdot \nabla u) \right\|_{L^2} ds\\
&& \quad \quad \quad \quad \quad \quad \quad \quad \quad \quad \quad +\int_0^t \left\|\Delta^{\frac{3}{2}} e^{-(t-s)\varepsilon \Delta^2} \Delta \left(a u^3+b u^2+cu \right) \right\|_{L^2} ds \\
&& \quad \quad  \le  Ct^{-\frac{1}{4}} \|u_0\|_{H^2} + C\int_0^t (t-s)^{-\frac{3}{4}} \left\|\Delta \left(a u^3+b u^2+cu \right) \right\|_{L^2} ds \\
&& \quad \quad \quad \quad \quad \quad \quad +C\int_0^t (t-s)^{-\frac{3}{4}} \|v \cdot \nabla u \|_{L^2} ds \\ 
&& \quad \quad \le  C t^{-\frac{1}{4}} \|u_0\|_{H^2}+ C \int_0^t (t-s)^{-\frac{3}{4}} ds \cdot \left(|a| \|u\|_{X_t}^3+|b| \|u\|_{X_T}^2+|c|\|u\|_{X_T} \right) \\
&&  \quad \quad \quad \quad \quad \quad \quad +C \int_0^t (t-s)^{-\frac{3}{4}}ds \cdot \|v\|_{L^\infty([0, \infty); L^\infty)} \|u\|_{X_T} \\
&& \quad \quad \le  C \bigg( t^{-\frac{1}{4}} \|u_0\|_{H^2}+ t^{\frac{1}{4}} \cdot \Big(|a| \|u\|_{X_T}^3+|\gamma_1| \|u\|_{X_T}^2 \\
&& \quad \quad \quad \quad \quad \quad \quad  \quad \quad \quad  \quad \quad  \quad  \quad \quad \quad \quad   +(1+\|v\|_{L^\infty([0, \infty); L^\infty)}) \|u\|_{X_T} \Big) \bigg).
\end{eqnarray*}
The desired estimate \eqref{20210729eq11} then follows from by multiplying $t^{\frac{1}{4}}$ on both sides of the above estimate.

Now we turn to prove the second part, that is
\begin{eqnarray} \label{20210729eq12}
 && \sup_{0 \le t \le T}  t^{\frac{1}{2}}\left\|\Delta^2 u \right\|_{L^2}  \nonumber \\ 
 && \quad \quad \le C \left( \|u_0\|_{H^2}+ t^{\frac{1}{2}} \left(  \|u\|_{X_T}^3+\|u\|_{X_T}^2+ \left(1+\|v\|_{L^\infty([0, \infty); W^{1, \infty})}\right) \|u\|_{X_T} \right) \right). 
\end{eqnarray} 
Note that
\begin{eqnarray*} 
\nabla \cdot \Delta (au^3+bu^2+cu)%
&=& 6a u \left(\nabla \cdot u \right) \Delta u+3a u^2 \nabla \cdot \Delta u  \nonumber \\
&& + 6a (\nabla \cdot u) |\nabla u|^2+12a u \sum_{i, j=1}^n \frac{\partial u}{\partial x_i} \frac{\partial^2 u}{\partial x_i \partial x_j}  \nonumber \\
&& +2b\sum_{i, j=1}^n \frac{\partial u}{\partial x_i} \frac{\partial^2 u}{\partial x_i \partial x_j}+2b \left(\nabla \cdot u \right) \Delta u \nonumber \\
&& +2b u \nabla \cdot \Delta u+c\nabla \cdot \Delta u. 
\end{eqnarray*}
This further gives
\begin{eqnarray*} 
&& \left\|\nabla \cdot \Delta (au^3+bu^2+cu) \right\|_{L^2}  \nonumber \\
&& \quad \quad \quad  \quad \le  18 |a| \left\|u \right\|_{L^\infty} \left\|\nabla u \right\|_{L^\infty} \|\Delta u\|_{L^2}+3|a| \|u\|_{L^\infty}^2 \left\| \Delta^{\frac{3}{2}} u \right\|_{L^2} \nonumber \\
&& \quad \quad \quad \quad \quad \quad +6|a| \left\|\nabla u \right\|_{L^\infty} \left\|\nabla u \right\|_{L^2}+4|b| \left\|\nabla u \right\|_{L^\infty} \|\Delta u \|_{L^2} \nonumber \\
&& \quad \quad \quad \quad \quad \quad +2|b| \|u\|_{L^\infty} \left\|\Delta^{\frac{3}{2}} u \right\|_{L^2}+|c|\left\|\Delta^{\frac{3}{2}} u \right\|_{L^2}
\end{eqnarray*}
This, together with \eqref{20210729eq20}, \eqref{20210729eq21} and the Gagliardo-Nirenberg's inequality 
\begin{equation} \label{20210729eq31} 
\|\nabla u \|_{L^\infty} \le C \left(\left\| \Delta^{\frac{3}{2}} u \right\|_{L^2}^{\frac{2}{3}} \|u\|_{L^2}^{\frac{1}{3}}+\|u\|_{L^2} \right) \le C \left\|\Delta^{\frac{3}{2}} u \right\|_{L^2},  
\end{equation} 
suggests 
\begin{equation} \label{20210729eq14} 
\left\|\nabla \cdot \Delta (au^3+bu^2+cu) \right\|_{L^2}  \le C \left(|a|\|u\|_{H^2}^2+(|a|+|b|) \|u\|_{H^2}+|c| \right) \left\|\Delta^{\frac{3}{2}} u \right\|_{L^2}. 
\end{equation}
On the other hand, we have 
\begin{eqnarray} \label{20210730eq41}
\left\|\nabla \cdot (v \cdot \nabla u) \right\|_{L^2} %
&=& \left\| \sum_{i=1}^n \left(\nabla \cdot v_i \right) u_{x_i}+\sum_{i=1}^n v_i \nabla \cdot u_{x_i} \right\|_{L^2} \nonumber \\
&\le& C\|v\|_{L^\infty([0, T); W^{1, \infty})} \|u\|_{H^2}. 
\end{eqnarray} 
Therefore, by \eqref{20210729eq14} and \eqref{20210730eq41}, we have
\begin{eqnarray*}
&& \left\|\Delta^2 \calT(u)(t) \right\|_{L^2} \le \left\|\Delta e^{-t\varepsilon \Delta^2} \left(\Delta u_0 \right) \right\|_{L^2}+\int_0^t \left\|\nabla \cdot \Delta e^{-(t-s)\varepsilon\Delta^2} \nabla (v \cdot \nabla u) \right\|_{L^2} ds \\
&&  \quad \quad \quad \quad \quad \quad \quad  +\int_0^t \left\|\nabla \cdot \Delta e^{-(t-s)\varepsilon \Delta^2} \left[ \nabla \cdot \Delta \left(a u^3+b u^2+cu \right) \right] \right\|_{L^2} ds \\
&&  \le C t^{-\frac{1}{2}}\|u_0\|_{H^2}+\int_0^t \left\|\nabla \cdot \Delta e^{-(t-s)\varepsilon \Delta^2}  \right\|_{L^2 \to L^2} \left\| \nabla \cdot \Delta \left(a u^3+b u^2+cu \right) \right\|_{L^2} ds \\
&& \quad \quad \quad \quad \quad \quad \quad +\int_0^t \left\|\nabla \cdot \Delta e^{-(t-s)\varepsilon \Delta^2}  \right\|_{L^2 \to L^2} \left\| \nabla \cdot (v \cdot \nabla u) \right\|_{L^2} ds \\
&& \le C t^{-\frac{1}{2}}\|u_0\|_{H^2}+ \int_0^t  (t-s)^{-\frac{3}{4}}  \cdot \left\| \nabla \cdot \Delta \left(a u^3+b u^2+cu \right) \right\|_{L^2} ds \\ 
&& \quad \quad \quad \quad \quad \quad \quad +\int_0^t (t-s)^{-\frac{3}{4}} \left\|\nabla \cdot \left(v \cdot \nabla u \right) \right\|_{L^2} ds \\
&& \le C t^{-\frac{1}{2}} \|u_0\|_{H^2} + \int_0^t (t-s)^{-\frac{3}{4}} \|v\|_{L^\infty([0, T); W^{1, \infty})} \|u\|_{H^2}  ds\\
&& \quad \quad +C \int_0^t (t-s)^{-\frac{3}{4}} s^{-\frac{1}{4}} \cdot  \left(\|u(s)\|_{H^2}^2+ \|u(s)\|_{H^2}+1 \right) \left(s^{\frac{1}{4}} \left\|\Delta^{\frac{3}{2}} u(s)  \right\|_{L^2} \right) ds \\
&& \le Ct^{-\frac{1}{2}}\|u_0\|_{H^2}+C \left(  \|u\|_{X_T}^3+\|u\|_{X_T}^2+ (1+t^{\frac{1}{4}} \|v\|_{L^\infty([0, \infty); W^{1, \infty})} )\|u\|_{X_T} \right) 
\end{eqnarray*} 
which implies the desired estimate \eqref{20210729eq12}. The estimate \eqref{20210729eq01} then clearly follows from \eqref{20210729eq11} and \eqref{20210729eq12}. 
\end{proof}

Combing Lemma \ref{20210728lem02} and Lemma \ref{20210728lem03}, we have the following result. 

\begin{lem} \label{20210730lem01}
For any $0<T \le 1$. The map $\calT: X_T \to X_T$ and there exists some $C_1=C_1(\varepsilon, a, b, c)>0$, such that
\begin{eqnarray*} 
&& \left\|\calT(u) \right\|_{X_T} \le C_1 \bigg( \|u_0\|_{H^2} \\ 
&& \quad \quad \quad \quad \quad \quad \quad \quad + T^{\frac{1}{2}} \left(  \|u\|_{X_T}^3+ \|u\|_{X_T}^2+  \left(1+\|v\|_{L^\infty([0, \infty); W^{1, \infty})}\right) \|u\|_{X_T} \right) \bigg).  
\end{eqnarray*}
\end{lem}

Next, we show that $\calT$ is a Lipschitz map on $X_T$. 

\begin{lem} \label{20210730lem02}
Let $0<T \le 1$. Then there exists a constant $C_2=C_2(\varepsilon, a, b, c)>0$, such that for $u_1, u_2 \in X_T$, 
\begin{eqnarray} \label{20210730eq01}
&& \|\calT(u_1)-\calT(u_2) \|_{X_T} \nonumber \\
&& \quad \quad \le  C_2 T^{\frac{1}{2}} \left[\left(\|u_1\|_{X_T}+\|u_2\|_{X_T}+1 \right)^2+\|v\|_{L^\infty([0, \infty); W^{1, \infty})}  \right]   \left\|u_1-u_2 \right\|_{X_T}. 
\end{eqnarray} 
\end{lem} 

\begin{proof}
To begin with, we note that by \eqref{20210729eq20} and \eqref{20210729eq21}, one can compute that
\begin{eqnarray} \label{20210730eq02} 
&&\left\|\Delta \left(au_1^3+b u_1^2+cu_1-\left(au_2^3+bu_2^2+cu_2\right) \right) \right\|_{L^2} \nonumber \\
&& \quad \quad \quad \quad \quad \quad \quad \quad \quad \le C\left(\|u_1\|_{H^2}+\|u_2\|_{H_2}+1 \right)^2 \|u_1-u_2\|_{H^2}, 
\end{eqnarray}
which is a consequence of the following estimates:
\begin{enumerate}
    \item [(1).] $
    \left\|\Delta(u_1^3-u_2^3) \right\|_{L^2} \le C \left(\|u_1\|_{H^2}^2+\|u_1\|_{H^2}\|u_2\|_{H^2}+\|u_2\|_{H^2}^2 \right) \|u_1-u_2\|_{H^2}$;  
    
    \medskip
    
    \item [(2).]
    $ \left\|\Delta(u_1^2-u_2^2) \right\|_{L^2} \le C \left(\|u_1\|_{H^2}+\|u_2\|_{H^2} \right) \|u_1-u_2\|_{H^2}$. 
\end{enumerate}
Using \eqref{20210730eq02}, we have for any $0 \le t \le T$, 
\begin{eqnarray} \label{20210730eq03}
&& \left\|\calT(u_1)(t)-\calT(u_2)(t) \right\|_{H^2} \nonumber \\
&& \quad \le \int_0^t \left\|\Delta e^{-(t-s)\varepsilon \Delta^2} \Delta \left(au_1^3+b u_1^2+cu_1-\left(au_2^3+bu_2^2+cu_2\right) \right) \right\|_{L^2} ds  \nonumber \\
&& \quad \quad \quad \quad +\int_0^t \left\|\Delta e^{-(t-s) \varepsilon \Delta^2} \left( v \cdot \nabla (u_1-u_2) \right) \right\|_{L^2} ds \nonumber \\ 
&& \quad \le C \int_0^t (t-s)^{-\frac{1}{2}} \left(\|u_1(s)\|_{H^2}+\|u_2(s)\|_{H_2}+1 \right)^2 \|u_1(s)-u_2(s)\|_{H^2} ds \nonumber \\
&& \quad \quad \quad \quad +\int_0^t (t-s)^{-\frac{1}{2}} \|v\|_{L^\infty([0, \infty); W^{1, \infty})} \|u_1(s)-u_2(s)\|_{H^2} ds \nonumber \\ 
&& \quad \le C \int_0^t (t-s)^{-\frac{1}{2}} ds \cdot \left(\|u_1\|_{X_T}+\|u_2\|_{X_T}+1 \right)^2  \left\|u_1-u_2 \right\|_{X_T}  \nonumber \\
&& \quad \quad \quad \quad +\int_0^t (t-s)^{-\frac{1}{2}} ds \cdot  \|v\|_{L^\infty([0, \infty); W^{1, \infty})} \|u_1-u_2\|_{X_T} ds \nonumber \\ 
&& \quad \le  Ct^{\frac{1}{2}} \cdot \left[\left(\|u_1\|_{X_T}+\|u_2\|_{X_T}+1 \right)^2+\|v\|_{L^\infty([0, \infty); W^{1, \infty})}  \right]    \left\|u_1-u_2 \right\|_{X_T}, 
\end{eqnarray}
and
\begin{eqnarray*}
&&\left\| \Delta^{\frac{3}{2}} \left(\calT(u_1)-\calT(u_2) \right)(t) \right\|_{L^2}  \\
&& \quad \le C \int_0^t (t-s)^{-\frac{3}{4}} \left(\|u_1(s)\|_{H^2}+\|u_2(s)\|_{H_2}+1 \right)^2 \|u_1(s)-u_2(s)\|_{H^2} ds \nonumber \\
&& \quad \quad \quad \quad + C \int_0^t (t-s)^{-\frac{3}{4}} \|v\|_{L^\infty([0, \infty); W^{1, \infty})} \|u_1(s)-u_2(s)\|_{H^2} ds \nonumber \\ 
&& \quad \le C \int_0^t (t-s)^{-\frac{3}{4}} ds \cdot \left(\|u_1\|_{X_T}+\|u_2\|_{X_T}+1 \right)^2  \left\|u_1-u_2 \right\|_{X_T}  \nonumber \\
&& \quad \quad \quad \quad + C \int_0^t (t-s)^{-\frac{3}{4}} ds \cdot  \|v\|_{L^\infty([0, \infty); W^{1, \infty})} \|u_1-u_2\|_{X_T} \nonumber \\ 
&& \quad \le  Ct^{\frac{1}{4}} \cdot \left[\left(\|u_1\|_{X_T}+\|u_2\|_{X_T}+1 \right)^2+\|v\|_{L^\infty([0, \infty); W^{1, \infty})}  \right]   \left\|u_1-u_2 \right\|_{X_T}, 
\end{eqnarray*}
which implies
\begin{eqnarray} \label{20210730eq04}
&& t^{\frac{1}{4}} \left\| \Delta^{\frac{3}{2}} \left(\calT(u_1)-\calT(u_2) \right)(t) \right\|_{L^2}  \nonumber \\
&& \quad \le  Ct^{\frac{1}{2}} \cdot \left[\left(\|u_1\|_{X_T}+\|u_2\|_{X_T}+1 \right)^2+\|v\|_{L^\infty([0, \infty); W^{1, \infty})}  \right]    \left\|u_1-u_2 \right\|_{X_T}. 
\end{eqnarray}

Next, we claim that
\begin{eqnarray} \label{20210730eq05}
&& t^{\frac{1}{2}} \left\|\Delta^2 \left(\calT(u_1)-\calT(u_2) \right)(t) \right\|_{L^2} \nonumber \\
&& \quad \le Ct^{\frac{1}{2}} \cdot \left[\left(\|u_1\|_{X_T}+\|u_2\|_{X_T}+1 \right)^2+\|v\|_{L^\infty([0, \infty); W^{1, \infty})}  \right]   \left\|u_1-u_2 \right\|_{X_T}. 
\end{eqnarray} 
Indeed, by \eqref{20210729eq20}, \eqref{20210729eq21}, \eqref{20210729eq31} and the definition of $X_T$, one can check that for each $0 \le s \le t$, 
\begin{eqnarray*} 
&& \left\|\nabla \cdot \Delta \left(au_1(s)^3+b u_1(s)^2+cu_1(s)-\left(au_2(s)^3+bu_2(s)^2+cu_2(s)\right) \right) \right\|_{L^2} \nonumber \\
&& \quad \quad \quad \quad \quad \quad \quad \quad \quad \quad\quad \quad \quad  \le C s^{-\frac{1}{4}} \left(\|u_1\|_{X_T}+\|u_2\|_{X_T}+1 \right)^2  \left\|u_1-u_2 \right\|_{X_T}, 
\end{eqnarray*} 
which implies 
\begin{eqnarray*}
&& \left\|\Delta^2 \left(\calT(u_1)-\calT(u_2) \right)(t) \right\|_{L^2} \\
&& \quad \quad \quad \le \int_0^t \left\| \nabla \cdot \Delta e^{-(t-s)\varepsilon \Delta^2} \nabla (v \cdot \nabla (u_1(s)-u_2(s))) \right\|_{L^2} \\
&& \quad \quad \quad \quad \quad \quad + \int_0^t \bigg\| \nabla \cdot \Delta  e^{-(t-s)\varepsilon \Delta^2} \nabla \cdot \Delta   \bigg [au_1(s)^3+b u_1(s)^2+cu_1(s) \\
&& \quad \quad \quad \quad \quad \quad \quad \quad  \quad\quad \quad \quad\quad \quad \quad  \quad  -\left(au_2(s)^3+bu_2(s)^2+cu_2(s)\right) \bigg] \bigg \|_{L^2} ds \\
&& \quad \quad \quad   \le C \int_0^t (t-s)^{-\frac{3}{4}} s^{-\frac{1}{4}} ds \cdot \left(\|u_1\|_{X_T}+\|u_2\|_{X_T}+1 \right)^2  \left\|u_1-u_2 \right\|_{X_T} \\
&& \quad \quad \quad \quad \quad \quad + \int_0^t (t-s)^{-\frac{3}{4}} ds \cdot \|v\|_{L^\infty([0, \infty); W^{1, \infty})} \|u_1-u_2\|_{X_T} \\ 
&& \quad \quad \quad   \le C \left[\left(\|u_1\|_{X_T}+\|u_2\|_{X_T}+1 \right)^2+t^{\frac{1}{4}} \|v\|_{L^\infty([0, \infty); W^{1, \infty})} \right) \left\|u_1-u_2 \right\|_{X_T}.
\end{eqnarray*}
This clearly gives \eqref{20210730eq05}. Finally, the desired estimate \eqref{20210730eq01} follows from \eqref{20210730eq03}, \eqref{20210730eq04} and \eqref{20210730eq05}. 
\end{proof}

\begin{proof} [Proof of Theorem \ref{20210728thm01}]
Let $C:=\max\{C_1, C_2\}$ and $M:=10C\|u_0\|_{H^2}$, where $C_1$ and $C_2$ are defined in Lemma \ref{20210730lem01} and Lemma \ref{20210730lem02}, respectively. Let further, $\B(0, M)$ be the ball in $X_T$ centered at origin with radius $M$ with 
$$
0<T<\min \left\{1, \frac{1}{16C^2(M^2+M+1+\|v\|_{L^\infty([0, \infty); W^{1, \infty})} )^2}\right\}. 
$$
It is then easy to see that
$$
\left\|\calT(u)\right\|_{X_T} \le M, \quad \forall u \in \B(0, M)
$$
and
$$
\left\|\calT(u_1)-\calT(u_2) \right\|_{X_T} \le \frac{1}{2} \|u_1-u_2\|_{X_T}, \quad \forall u_1, u_2 \in \B(0, M). 
$$
An application of Banach contraction mapping theorem yields Theorem \ref{20210728thm01}.  
\end{proof}

As a consequence, we have the following corollary.

\begin{cor} \label{20210731cor01} 
Under the assumption of Theorem \ref{20210728thm01}, if $\tilde{T}$ is the maximal time of existence of the mild solution $u=\calT(u)$, then 
$$
\limsup_{t \to \tilde{T}_{-}} \|u(t)\|_{H^2}=\infty. 
$$
Otherwise, $\tilde{T}=\infty$. 
\end{cor}

\begin{proof}
The proof for the above result is standard. For example, one can consult \cite[Corollary 2.7]{FM19} (see, also \cite[Corollary 2.6]{FHX21}). 
\end{proof}

\begin{rem} \label{20211015rem02} 
\begin{enumerate}
\item [(1). ] Corollary \ref{20210731cor01} can be viewed as a supplement to \cite[Theorem 3.1]{ES86}, which asserts that in the non-advective case (namely, ${\bf v} \equiv 0$), if the Landau-Ginzurg free energy  of the initial data is sufficiently negative, then the $H^2$-energy (more precisely, $L^2$-energy) of the solution will blow up in finite time (see, Remark \ref{20211026rem01}); while Corollary \ref{20210731cor01} gives that if the mild solution with $H^2$ initial data exists in finite time, then its $H^2$-norm has to blow up. 

\medskip

\item [(2).] Via a standard approximation argument, one can also check the mild solution $u$ constructed in Theorem \ref{20210728thm01} is also a \emph{weak solution} on $[0, T)$ (see, e.g., \cite[Proposition 2.8]{FM19} and \cite[Proposition 2.9]{FHX21}). Here, a function $\widetilde{u} \in L^\infty([0, T]; H^2)$ is called a \emph{weak solution} of \eqref{maineq} on $[0, T]$ with initial data $\widetilde{u}_0 \in H^2$ if for all $\phi \in C_c^\infty ([0, T) \times \T^2)$, 
\begin{eqnarray*}
&& \int_{\T^2} \widetilde{u}_0\phi(0)dxdy+\int_0^T \int_{\T^2} \widetilde{u} \partial_t \phi dxdydt =\int_0^T \varphi(v \cdot \nabla \widetilde{u}) dxdydt \\
&& \quad \quad \quad +\varepsilon \int_0^T \int_{\T^2} \phi \Delta^2 \widetilde{u} dxdydt-\int_0^T \int_{\T^2} \phi \Delta \left(a u^3+bu^2+cu \right) dxdydt
\end{eqnarray*}
and $\partial_t \widetilde{u} \in L^2([0, T]; H^{-2})$. Moreover, $u$ also satisfies the following energy identity: for any $0 \in [0, T)$, 
\begin{eqnarray} \label{energyid}
&& \|u(0)\|_{L^2}+2\int_0^T \int_{\T^2} u \Delta\left(a u^3+b u^2+cu \right) dxdydt \nonumber \\
&& \quad \quad \quad \quad \quad \quad \quad =\|u(t)\|_{L^2}^2+2\varepsilon \int_0^t \|\Delta u(s)\|_{L^2}^2 ds. 
\end{eqnarray} 

\medskip

\item [(3).] We claim  for each $t>0$, $\partial_t u(t, \cdot) \in L^2(\T^2)$. Indeed, by standard argument from spectral theory, one can verify that in the weak sense, 
\begin{equation} \label{20210801eq01}
\quad \quad \quad  \partial_t u(t)=- \varepsilon \Delta^2 u(t)+\Delta(a u^3(t)+b u(t)^2+cu(t))-{\bf v} \cdot \nabla u(t). 
\end{equation} 
Note that the last two terms in the right hand side of \eqref{20210801eq01} belongs to $L^2$, while for the first term, by Theorem \ref{20210728thm01}, we already know that
$$
t^{\frac{1}{2}} \left\|\Delta^2 u \right\|_{L^2}<\infty.
$$
The desired claim is then immediate. 
\end{enumerate} 
\end{rem}

\section{Bootstrap assumptions} \label{bootassumption} 

We now turn to the proof of Theorem \ref{mainthm}, which we recall deals with the global existence of the mild solutions to the following \emph{unstable} Cahn-Hilliard equation:
$$
\begin{cases}
u_t+v_1(y) \partial_x u+\varepsilon \gamma \Delta^2 u= \gamma \Delta(a u^3+ b u^2)  \quad & \quad  \textrm{on} \quad \T^2;  \\
\\
u \ \textrm{periodic} \quad & \quad \textrm{on} \quad  \partial \T^2, 
\end{cases}
$$
with $\varepsilon, \gamma>0, a<0, b \in \R$, $v_1 \in W^{1, \infty}(\T)$ and the initial condition $u_0 \in H^2_0(\T^2)$. The goal of this section is to establish the \emph{bootstrap assumptions} to the above problem. 

 Let $u$ be the unique mild solution on $[0, T]$ to the problem \eqref{maineq02} constructed in Theorem \ref{20210728thm01} and recall that 
\begin{equation} \label{20210827eq17}
\langle u \rangle(t, y):=\int_{\T} u(t, x, y) dx \quad \textrm{and} 
\quad u_{\notparallel} (t, x, y):=u(t, x, y)-\langle u \rangle (t, y), 
\end{equation} 
Note that both $\langle u \rangle$ and $u_{\notparallel}$ have mean zero. Using Theorem \ref{20210728thm01}, we have
$$
\langle u \rangle \in L_{loc}^2 \left((0, T); H^4(\T) \right) \cap C\left([0, T); H^2(\T) \right) 
$$
and
$$
u_{\notparallel} \in L_{loc}^2 \left((0, T); H^4(\T^2) \right) \cap C\left([0, T); H^2(\T^2) \right). 
$$
Moreover, using \eqref{maineq02}, one can check that $\langle u \rangle$ and $u_{\notparallel}$ solve the following coupled system:
\begin{equation} \label{20210822eq03}
\partial_t \langle u \rangle+\varepsilon \gamma \partial_y^4 \langle u \rangle= \gamma \int_{\T} \Delta \left[a\left(\langle u \rangle+u_{\notparallel} \right)^3+b \left(\langle u \rangle+u_{\notparallel} \right)^2\right] dx 
\end{equation}
and
\begin{eqnarray} \label{20210822eq04}
&& \partial_t u_{\notparallel}+v_1(y)\partial_x u_{\notparallel}+\varepsilon \gamma \Delta^2 u_{\notparallel}=\gamma \Delta \left[a\left(\langle u \rangle+u_{\notparallel} \right)^3+b \left(\langle u \rangle+u_{\notparallel} \right)^2\right]  \nonumber \\
&& \quad \quad \quad \quad \quad \quad \quad \quad  \quad \quad \quad  \quad  \quad \quad  \quad - \gamma \int_{\T} \Delta \left[a\left(\langle u \rangle+u_{\notparallel} \right)^3+b \left(\langle u \rangle+u_{\notparallel} \right)^2\right] dx. 
\end{eqnarray} 
Let $\calS_t$ be the solution operator form $0$ to time $t \ge 0$ for the advection-hyperdiffusion equation:
$$
\partial_t g+v_1(y)\partial_x g+\varepsilon \gamma \Delta^2 g=0.
$$
Namely, $\calS_t=e^{-tH_{\gamma}}$, where we recall that $H_\gamma$ is defined in \eqref{20210822eq05}. Then by the Duhamel's formula, we have for any $0 \le s \le t$, 
\begin{eqnarray} \label{20210823eq01}
u_{\notparallel}(t)%
&=& \calS_{t-s} \left(u_{\notparallel}(s) \right)  
 +\gamma \int_s^t \calS_{t-\tau} \left( \Delta \left[a\left(\langle u \rangle+u_{\notparallel} \right)^3+b \left(\langle u \rangle+u_{\notparallel} \right)^2\right] \right) d\tau \nonumber \\
&& \quad \quad \quad -\gamma \int_s^t  \calS_{t-\tau} \left(\int_{\T}  \Delta \left[a\left(\langle u \rangle+u_{\notparallel} \right)^3+b \left(\langle u \rangle+u_{\notparallel} \right)^2\right] dx\right) d\tau. 
\end{eqnarray}

We have the following results.

\begin{lem} \label{20210827lem01}
There exists a sufficiently small time $0<t_1<T$, which only depends on $\varepsilon, a, b, \gamma, v_1, \|u_{\notparallel}(0)\|_{L^2}$ and $\|u\|_{L^2([0, T]; H^2)}$, such that for any $0 \le s \le t \le t_1$, one has
\begin{equation} \label{20210827eq10}
\|u_{\notparallel}(t)\|_{L^2} \le C e^{-\frac{ \lambda_{\gamma}(t-s)}{4}} \|u_{\notparallel}(s)\|_{L^2}, 
\end{equation} 
where $\lambda_{\gamma}$ is defined in \eqref{20210822eq02}. 
\end{lem}

\begin{proof}
Taking $L^2$-norm on both sides of \eqref{20210823eq01} , we have
\begin{eqnarray*}
&&\|u_{\notparallel}(t)\|_{L^2}\le  \left\|\calS_{t-s} \left(u_{\notparallel}(s) \right) \right\|_{L^2} \nonumber \\
&& \quad \quad \quad  +|a| \gamma \int_s^t \left\|\calS_{t-\tau} \left( \Delta \left(\langle u \rangle+u_{\notparallel} \right)^3-\int_{\T}\Delta \left(\langle u \rangle+u_{\notparallel} \right)^3 dx  \right)\right\|_{L^2} d\tau \nonumber \\
&& \quad \quad \quad +|b|\gamma \int_s^t  \left\|\calS_{t-\tau} \left(  \Delta \left(\langle u \rangle+u_{\notparallel} \right)^2-\int_{\T} \Delta \left(\langle u \rangle+u_{\notparallel} \right)^2 dx\right) \right\|_{L^2} d\tau.
\end{eqnarray*}
By Proposition \ref{20210823prop01} and the fact that $\left\|\calS_t \right\|_{L^2 \to L^2} \le 1$, we can further bound $\|u_{\notparallel}(t)\|_{L^2}$ by 
\begin{equation} \label{20210823eq10}
10e^{-\lambda_{\gamma}(t-s)}\|u_{\notparallel}(s)\|_{L^2}+ |a|\gamma I_1+ |b|\gamma I_2, 
\end{equation} 
where 
\begin{eqnarray*}
I_1%
&:=&  \int_s^t \left\|\Delta \left(\langle u \rangle+u_{\notparallel} \right)^3-\int_{\T} \Delta \left(\langle u \rangle+u_{\notparallel} \right)^3 dx \right\|_{L^2} d\tau \\
&=& \int_s^t \left\|\Delta(u^3)-\int_{\T} \Delta (u^3) dx \right\|_{L^2} d\tau
\end{eqnarray*}
and
\begin{eqnarray*}
I_2%
&:=&  \int_s^t \left\|\Delta \left(\langle u \rangle+u_{\notparallel} \right)^2-\int_{\T} \Delta \left(\langle u \rangle+u_{\notparallel} \right)^2 dx \right\|_{L^2} d\tau \\
&=& \int_s^t \left\|\Delta(u^2)-\int_{\T} \Delta (u^2) dx \right\|_{L^2} d\tau.
\end{eqnarray*}

\medskip

\textit{Estimate of $I_1$}. Note that
$$
\Delta(u^3)=3u^2 \Delta u+6u|\nabla u|^2.
$$
Thus, 
\begin{eqnarray} \label{20210827eq01}
\|\Delta(u^3)\|_{L^2}%
&\le& C \left[ \|u^2 \Delta u\|_{L^2}+\|u |\nabla u|^2\|_{L^2} \right] \nonumber \\
&\le& C \left[\|u\|_{L^\infty}^2 \|\Delta u\|_{L^2}+\|u\|_{L^2} \|\nabla u\|_{L^4}^2 \right] \nonumber \\
&\le&  C \left[\|\Delta u\|_{L^2}^2 \|u\|_{L^2}+\|u\|_{L^2}^{\frac{3}{2}} \|\Delta u\|_{L^2}^{\frac{3}{2}} \right] \nonumber \\
&\le& C\|u\|_{H^2}^3, 
\end{eqnarray}
where in the above estimates, we have used the Gagliardo–Nirenberg's inequalities in 2D:
$$
\|u\|_{L^\infty} \le C \|\Delta u\|_{L^2}^{\frac{1}{2}} \|u\|_{L^2}^{\frac{1}{2}} 
 \quad \textrm{and} \quad 
\left\|\nabla u \right\|_{L^4} \le C \|\Delta u\|_{L^2}^{\frac{3}{4}} \|u\|_{L^2}^{\frac{1}{4}}.
$$
Using \eqref{20210827eq01}, we have
$$
I_2 \le C \int_s^t \left\|\Delta(u^3) \right\|_{L^2} d\tau \le C \int_s^t \|u\|_{H^2}^3 d\tau \le C (t-s) \|u\|^3_{L^\infty\left(\left[0, T\right]; H^2 \right)},
$$
where in the last estimate above, we have used the fact that $0 \le s \le t \le T$. 

\medskip

\textit{Estimate of $I_2$.} The estimate of $I_2$ is similar to the one  of $I_1$. We first note that 
$$
\Delta(u^2)=2u\Delta u+2|\nabla u|^2,
$$
and hence we have 
\begin{eqnarray} \label{20210827eq20} 
\|\Delta(u^2)\|_{L^2}%
&\le& C\left[\|u\Delta u\|_{L^2}+\|\nabla u\|^2_{L^4}\right] \nonumber  \\
&\le& C \left[\|u\|_{L^\infty} \|\Delta u\|_{L^2}+\|\nabla u\|_{L^4}^2 \right] \nonumber \\
&\le& C\|\Delta u\|_{L^2}^{\frac{3}{2}} \|u\|_{L^2}^{\frac{1}{2}} \le C\|u\|_{H^2}^2.
\end{eqnarray}
This implies
$$
I_3 \le C \int_s^t \|\Delta(u^2)\|_{L^2} d\tau \le C\int_s^t \|u\|_{H^3}^2 d\tau \le C(t-s) \|u\|^2_{L^\infty\left(\left[0, T\right]; H^2 \right)}. 
$$

\medskip

Finally, combining all the estimates of $I_1$ and $I_2$ with \eqref{20210823eq10}, we see that 
\begin{eqnarray*}
\|u_{\notparallel}(t)\|_{L^2}%
&\le&  10e^{-\lambda_{\gamma}(t-s)} \|u_{\notparallel}(s)\|_{L^2} +C(t-s) \gamma \bigg[ \|u\|^3_{L^\infty([0, T]; H^2)} \\
&& \quad \quad \quad \quad \quad \quad \quad +|a|\|u\|_{L^\infty([0, T]; H^2)}^2+|b|\|u\|_{L^\infty([0, T]; H^2)} \bigg]. 
\end{eqnarray*}
The desired estimate \eqref{20210827eq10} then follows from by choosing $t_1$ sufficiently small and the continuity of $\|u(t)\|_{L^2}$ at $t=0$. 
\end{proof}

\begin{lem} \label{20210827lem02}
There exists a sufficiently small time $0<t_2<T$, such that for any $0 \le s \le t_2$, one has
\begin{equation} \label{20210827eq15}
\varepsilon \gamma \int_s^t \left\|\Delta u_{\notparallel}(\tau) \right\|_{L^2}^2 d\tau \le C \|u_{\notparallel}(s)\|_{L^2}^2.
\end{equation} 
Here, $t_2$ only depends on $\varepsilon, a, b, \gamma, v_1, \|u_{\notparallel}(0)\|_{L^2}$ and $\|u\|_{L^\infty([0, 1]; H^2)}$.
\end{lem}

\begin{proof}
Multiply \eqref{20210822eq04} by $u_{\notparallel}$ on both sides and then integrating by parts, we see that for any $t>0$,
\begin{eqnarray} \label{20210827eq21} 
&&\frac{1}{2} \frac{d}{dt} \|u_{\notparallel}(t)\|_{L^2}^2+\varepsilon \gamma \|\Delta u_{\notparallel}(t)\|_{L^2}^2 = a\gamma \int_{\T^2} u_{\notparallel}(t) \Delta(u^3)(t)dxdy \nonumber \\
&& \quad \quad \quad \quad  +b\gamma \int_{\T^2} u_{\notparallel}(t) \Delta (u^2)(t)dxdy  -a\gamma \int_{\T} \left(\int_{\T} u_{\notparallel}(t) dx \right) \left(\int_{\T} \Delta (u^3) dx \right) dy \nonumber \\
&& \quad \quad \quad \quad - b \gamma \int_{\T} \left(\int_{\T} u_{\notparallel}(t) dx \right) \left(\int_{\T} \Delta (u^2) dx \right) dy \nonumber \\
&& \quad = a\gamma \int_{\T^2} u_{\notparallel}(t) \Delta(u^3)(t)dxdy +b\gamma \int_{\T^2} u_{\notparallel}(t) \Delta (u^2)(t)dxdy \nonumber  \\
&& \quad \le  |a| \gamma \|u_{\notparallel}\|_{L^2} \|\Delta(u^3)\|_{L^2} +|b|\gamma \|u_{\notparallel}\|_{L^2} \|\Delta(u^2)\|_{L^2}, 
\end{eqnarray}
where in the last second line of the above estimate, we have used the fact that $\int_{\T} u_{\notparallel}(t) dx=0$.

Note that by \eqref{20210827eq01} and \eqref{20210827eq20} respectively, we have
$$
\|u_{\notparallel}\|_{L^2} \|\Delta(u^3)\|_{L^2} \le C \|u_{\notparallel}\|_{L^2} \|u\|_{H^2}^3
 \quad \textrm{and} \quad 
 \|u_{\notparallel}\|_{L^2} \|\Delta(u^2)\|_{L^2} \le C \|u_{\notparallel}\|_{L^2} \|u\|_{H^2}^2. 
$$
Hence, 
$$
\textrm{RHS of \eqref{20210827eq21}} \le C\gamma \left[  \|u_{\notparallel}\|_{L^2} \|u\|_{H^2}^3+ \|u_{\notparallel}\|_{L^2} \|u\|_{H^2}^2 \right].
$$
This gives 
\begin{eqnarray*}
\frac{\|u_{\notparallel}(t)\|_{L^2}^2}{2}+\varepsilon \gamma \int_s^t \|\Delta u_{\notparallel}(\tau) \|_{L^2}^2 d\tau%
&\le& \frac{\|u_{\notparallel}(s)\|_{L^2}^2}{2}+C\gamma \int_s^t \|u_{\notparallel}\|_{L^2} \|u\|_{H^2}^3 d\tau \\
&&+C\gamma \int_s^t \|u_{\notparallel}\|_{L^2} \|u\|_{H^2}^2 d\tau,
\end{eqnarray*}
and hence
\begin{eqnarray*}
\varepsilon \gamma \int_s^t \|\Delta u_{\notparallel}(\tau)\|_{L^2}^2 d\tau%
&\le& \frac{\|u_{\notparallel}(s)\|_{L^2}^2}{2}+ C\gamma (t-s) \Big( \|u\|^4_{L^\infty([0, 1]; H^2)}+ \\
&& \quad \quad \quad \quad \quad \quad  \quad \quad \quad  \|u\|^3_{L^\infty([0, 1]; H^2)} \Big). 
\end{eqnarray*} 
The desired estimate \eqref{20210827eq15} then follows from by choosing $t_2$ sufficiently small and the continuity of $\|u(t)\|_{L^2}$ at $t=0$. 
\end{proof}

The estimates \eqref{20210827eq10} and \eqref{20210827eq15} suggests for all sufficiently small time $t \ge s \ge 0$, we can assume that 
\begin{enumerate}
    \item [(1).] $\|u_{\notparallel}(t)\|_{L^2} \le 20 e^{-\frac{\lambda_{\gamma}(t-s)}{4}} \|u_{\notparallel}(s)\|_{L^2}$;
    \item [(2).] $\varepsilon \gamma \int_s^t \|\Delta u_{\notparallel}(\tau)\|_{L^2}^2 d\tau  \le 10 \|u_{\notparallel}(s)\|_{L^2}^2$.
\end{enumerate}
Finally, let $t_0>0$ is the maximal time such that estimates above hold on $[0, t_0]$, and we refer the two estimates above together with $t \in [0, t_0]$ as the \emph{bootstrap assumptions}. 

\begin{rem}
Let us clarify the role of the maximal time $t_0$ in the bootstrap assumptions above. 
The estimates in Lemmas \ref{20210827lem01} and \ref{20210827lem02} show that the two bootstrap assumptions hold on a sufficiently short time interval; hence $t_0>0$ is well-defined. The main goal of the rest of the paper is to prove, under the assumptions of Theorem \ref{mainthm}, that in fact $t_0=+\infty$.

The argument is by contradiction. Suppose that $t_0<+\infty$. Then the main task is to show that, on $[0,t_0]$, the bootstrap estimates can be improved: the constants $20$ and $10$ in the two bootstrap assumptions above can be replaced by strictly smaller constants. By the continuity of the corresponding norms and the absolute continuity of the time-integrated dissipation term at $t_0$, these improved estimates imply that the original bootstrap assumptions (with constants $20$ and $10$) remain valid on a slightly larger interval $[0,t_0+\epsilon]$, for some $\epsilon>0$. This contradicts the maximality of $t_0$, and therefore one must have $t_0=+\infty$. We refer the reader to the proof of Theorem \ref{mainthm} in Section \ref{bootest}  for the precise closing argument.
\end{rem}

\medskip 

Before entering the technical details, we summarize our proof of the bootstrap argument in Figure \ref{fig:bootstrap-roadmap}. 
The bootstrap assumptions first yield uniform control of the shear-invariant component through Proposition \ref{avgprop}. This control is then used in Propositions \ref{bootestprop1} and \ref{bootestprop3}, which improve the two bootstrap estimates. The strict improvement closes the bootstrap and forces $t_0=+\infty$.

\begin{figure}[htbp]
\centering
\begin{tikzpicture}[
    font=\footnotesize, scale=.95,
    box/.style={
        rectangle,
        rounded corners,
        draw=black!60,
        thick,
        align=center,
        text width=0.78\textwidth,
        inner sep=6pt
    },
    arrow/.style={
        -{Latex[length=2.4mm]},
        thick
    }
]

\node[box] (boot) at (0,0) {
\textbf{Proof by contradiction: bootstrap assumptions on $[0,t_0]$ for some $t_0$ being finite}\\[1mm]
The two bootstrap estimates for $u_{\notparallel}$ hold with constants $20$ and $10$.
};

\node[box] (avg) at (0,-1.75) {
\textbf{Step 1: Control of the shear-invariant component}\\[1mm]
Proposition \ref{avgprop} gives uniform bounds for $\langle u\rangle$ under the bootstrap assumptions.
};

\node[box] (diss) at (0,-3.68) {
\textbf{Step 2: Improvement of the integrated dissipation estimate}\\[1mm]
Proposition \ref{bootestprop1} improves the second bootstrap estimate, replacing the constant $10$ by the sharper constant $5$.
};

\node[box] (decay) at (0,-5.73) {
\textbf{Step 3: Improvement of the decay estimate}\\[1mm]
Proposition \ref{bootestprop3} improves the first bootstrap estimate, replacing the constant $20$ by the sharper constant $15$. 
};

\node[box] (close) at (0,-7.60) {
\textbf{Closing the bootstrap}\\[1mm]
The improved estimates are strictly stronger than the original bootstrap assumptions.
};

\node[box] (conclusion) at (0,-9.50) {
\textbf{Conclusion}\\[1mm]
By continuity, the original bootstrap assumptions extend beyond $t_0$, contradicting the maximality of $t_0$. Therefore the maximal time $t_0$ is infinite.
};

\draw[arrow] (boot) -- (avg);
\draw[arrow] (avg) -- (diss);
\draw[arrow] (diss) -- (decay);
\draw[arrow] (decay) -- (close);
\draw[arrow] (close) -- (conclusion);

\end{tikzpicture}
\caption{Roadmap of the bootstrap argument.}
\label{fig:bootstrap-roadmap}
\end{figure}

\section{Uniform bounds of $\langle u \rangle$}

As a consequence of the bootstrap assumptions, we show that the terms 
$$
\left\|\langle u\rangle \right\|_{L^2} \qquad \textrm{and} \qquad \varepsilon \gamma \int_0^t \left\|\partial_y^2 \langle u \rangle  \right\|_{L_y^2}^2 d\tau
$$
can be bounded uniformly when $\gamma$ is sufficiently small and the $L^2$-energy of $\langle u_0 \rangle$ is sufficiently small. Motivated by \cite{FHXZ21}, we show such uniform bounds by applying a version of \emph{small energy method} under the setting of shear flows. 

\begin{lem} [A prior estimate] \label{priopest} 
Assume the bootstrap assumptions, $|a|<1$ and moreover, there exists a positive number $0<\widetilde{t_0} \le t_0$, such that
\begin{equation} \label{20211016eq01}
\left\| \langle u \rangle (t, \cdot) \right\|_{L_y^2} \le K, \quad \textrm{for all} \quad t \in [0, \widetilde{t_0}]
\end{equation}
where 
\begin{equation} \label{20211016eq02}
K \le \min \left\{1, \left(\frac{\varepsilon \lambda_1^2}{4{\bf B_1}} \right)^{\frac{3}{8}} \right\},
\end{equation} 
Here, $t_0$ is the maximal time defined in the bootstrap assumptions, $\lambda_1$ is the smallest positive eigenvalue of $-\Delta$ on $\T$ and ${\bf B_1}>1$ is an absolute constant which only depends on $\varepsilon,  b$ and any dimensional constants. 

Then the following estimate holds: there exists some constant ${\bf B_2}>1$, which only depends on $\varepsilon, b$ and any dimensional constants, such that 
\begin{eqnarray*}
\left\|\langle u \rangle \left(\widetilde{t_0} \right)\right\|_{L_y^2}^2 %
&\le&  {\bf B_2} \exp \left( {\bf B_2} \|u_{\notparallel}(0)\|_{L^2}^4 \right) \bigg( a^2 \left\|u_{\notparallel}(0) \right\|_{L^2}^6 \nonumber \\
&&  \quad \quad \quad \quad + \left(\frac{\gamma}{\lambda_\gamma} \right)^{\frac{1}{2}} \left\|u_{\notparallel}(0) \right\|_{L^2}^4+\left\|\langle u \rangle(0) \right\|_{L_y^2}^2 \bigg).
\end{eqnarray*}
\end{lem}

\begin{proof} 
We begin with recalling that the constant $B$ that we are going to use in the proof might change line by line, but will only depend on $\varepsilon, b$ and any dimensional constants. It is important that $B$ is \emph{independent} of the choice of $a$ and $\gamma$. 

Multiplying $\langle u \rangle$ on both sides of \eqref{20210822eq03} and then integration by parts, we have
\begin{eqnarray} \label{20210829eq500} 
&& \frac{1}{2} \frac{d}{dt} \|\langle u \rangle \|_{L_y^2}^2+ \varepsilon \gamma \left\|\partial_y^2 \langle u \rangle \right\|_{L_y^2}^2 \nonumber \\
&& \quad \quad \quad =\gamma \int_{\T} \langle u\rangle \left(\int_{\T} \Delta \left[a\left(\langle u \rangle+u_{\notparallel} \right)^3+b \left(\langle u \rangle+u_{\notparallel} \right)^2\right] dx \right) dy \nonumber \\
&& \quad \quad \quad =\gamma \int_{\T} \langle u \rangle \left(\int_{\T} \partial_y^2 \left[a\left(\langle u \rangle+u_{\notparallel} \right)^3+b \left(\langle u \rangle+u_{\notparallel} \right)^2\right] dx \right) dy \nonumber \\
&& \quad \quad \quad =\gamma \int_{\T^2} \partial_y^2 \langle u \rangle \cdot  \left[ a(\langle u \rangle+u_{\notparallel})^3+b \left( \langle u \rangle+u_{\notparallel} \right)^2 \right]dxdy \nonumber \\
&& \quad \quad \quad :=J_1+J_2+J_3+J_4, 
\end{eqnarray}
where
$$
J_1:=\gamma \int_{\T} \partial_y^2 \langle u \rangle \left( a\langle u \rangle^3+b \langle u \rangle^2 \right)dy, \quad  J_2:=3a\gamma \int_{\T^2} \partial_y^2 \langle u \rangle \langle u \rangle u_{\notparallel}^2 dxdy, 
$$
$$
J_3:=a\gamma \int_{\T^2} \partial_y^2 \langle u \rangle u_{\notparallel}^3 dxdy
 \quad \textrm{and} \quad 
J_4:=b\gamma \int_{\T^2} \partial_y^2 \langle u \rangle u_{\notparallel}^2 dxdy.
$$
Note that here we have used the fact that
$$
\int_{\T^2} \partial_y^2 \langle u \rangle \langle u \rangle^2 u_{\notparallel}dxdy=\int_{\T^2}  \partial_y^2 \langle u \rangle \langle u \rangle u_{\notparallel}dxdy=0. 
$$

\medskip

\textit{Estimate of $J_1$.}  By Young's inequality and the Gagliardo-Nirenberg’s inequalities in 2D:
$$
\left\| \partial_y \langle u \rangle \right\|_{L_y^4} \le B \left\|\partial_y^2 \langle u \rangle \right\|_{L_y^2}^{\frac{5}{8}} \left\| \langle u \rangle \right\|_{L_y^2}^{\frac{3}{8}} 
 \quad \textrm{and} \quad 
\left\|\langle u \rangle \right\|_{L_y^4} \le B \left\| \partial_y^2 \langle u \rangle \right\|_{L_y^2}^{\frac{1}{8}} \left\| \langle u \rangle \right\|_{L_y^2}^{\frac{7}{8}},
$$
we have 
\begin{eqnarray*}
J_1%
&=& -\gamma \int_{\T} \left(\partial_y \langle u \rangle \right)^2  \cdot \left(3a\langle u \rangle^2+2b \langle u \rangle \right) dy \\
&\le& 3|a| \gamma \left\| \partial_y \langle u \rangle \right\|_{L_y^4}^2 \left\| \langle u \rangle \right\|_{L_y^4}^2+ 2|b| \gamma \left\|\partial_y \langle u \rangle \right\|_{L_y^4}^2 \left\| \langle u \rangle \right\|_{L_y^2} \\
&=& B|a|\gamma \left\|\partial_y^2 \langle u \rangle \right\|_{L_y^2}^{\frac{3}{2}} \left\| \langle u \rangle \right\|_{L_y^2}^{\frac{5}{2}} +B\gamma \left\|\partial_y^2 \langle u \rangle \right\|_{L_y^2}^{\frac{5}{4}} \left\| \langle u \rangle \right\|_{L_y^2}^{\frac{7}{4}} \\
&\le& \frac{\varepsilon \gamma}{8} \left\|\partial_y^2 \langle u \rangle \right\|_{L_y^2}^2+B(1+a^2)\gamma \left(\left\|\langle u\rangle\right\|_{L_y^2}^{10}+\left\|\langle u\rangle\right\|_{L_y^2}^{\frac{14}{3}} \right) \\
&\le& \frac{\varepsilon \gamma}{8} \left\| \partial_y^2 \langle u \rangle \right\|_{L_y^2}^2+B\gamma \cdot \left(K^8+K^{\frac{8}{3}} \right) \left\|\langle u \rangle \right\|_{L_y^2}^2 \\
&\le& \frac{\varepsilon \gamma}{8} \left\| \partial_y^2 \langle u \rangle \right\|_{L_y^2}^2+B \gamma K^{\frac{8}{3}}  \left\|\langle u \rangle \right\|_{L_y^2}^2 
\end{eqnarray*}
where in the second estimate above, we have used the assumption $|a|<1$ and in the second last estimate, we have used the assumption \eqref{20211016eq01}.  

\medskip

\textit{Estimate of $J_2$.} By Young's inequality, we have
\begin{eqnarray*}
J_2%
&\le& 3|a| \gamma \left\|\partial_y^2 \langle u \rangle\right\|_{L_y^2} \left(\int_{\T^2} \langle u \rangle^2 u_{\notparallel}^4 dxdy \right)^{\frac{1}{2}} \\ 
&\le& \frac{\varepsilon \gamma}{8} \left\|\partial_y^2 \langle u \rangle\right\|_{L_y^2}^2+Ba^2\gamma \int_{\T^2} \langle u \rangle^2 u_{\notparallel}^4 dxdy \\
&\le& \frac{\varepsilon \gamma}{8} \left\|\partial_y^2 \langle u \rangle\right\|_{L_y^2}^2+B\gamma \|u_{\notparallel}\|_{L^\infty}^4 \|\langle u \rangle\|_{L_y^2}^2 \\
&\le& \frac{\varepsilon \gamma}{8} \left\|\partial_y^2 \langle u \rangle\right\|_{L_y^2}^2+B\gamma \|\Delta u_{\notparallel}\|_{L^2}^2 \|u_{\notparallel}\|_{L^2}^2 \|\langle u \rangle\|_{L_y^2}^2, 
\end{eqnarray*}
where in the last estimate above, we have used the Gagliardo–Nirenberg's inequality in 2D:
$$
\|u_{\notparallel}\|_{L^\infty} \le B \|\Delta u_{\notparallel}\|_{L^2}^{\frac{1}{2}} \|u_{\notparallel}\|_{L^2}^{\frac{1}{2}}. 
$$

\medskip

\textit{Estimate of $J_3$.} By the Young's inequality and Gagliardo–Nirenberg's inequality in 2D:
$$
\|u_{\notparallel}\|_{L^6} \le B \|\Delta u_{\notparallel}\|_{L^2}^{\frac{1}{3}} \|u_{\notparallel}\|_{L^2}^{\frac{2}{3}}, 
$$
we have
\begin{eqnarray*}
J_3%
&\le& |a| \gamma \left\|\partial_y^2 \langle u \rangle \right\|_{L_y^2} \left\|u_{\notparallel}^3 \right\|_{L^2} \le \frac{\varepsilon \gamma}{8} \left\|\partial_y^2 \langle u \rangle\right\|_{L_y^2}^2+Ba^2\gamma \int_{\T^2} u_{\notparallel}^6 dxdy \\
&\le& \frac{\varepsilon \gamma}{8} \left\|\partial_y^2 \langle u \rangle\right\|_{L_y^2}^2+Ba^2 \gamma \|\Delta u_{\notparallel}\|_{L^2}^2 \|u_{\notparallel}\|_{L^2}^4.
\end{eqnarray*}

\medskip

\textit{Estimate of $J_4$.} The estimate of $J_4$ is similar to the one of $J_3$:
\begin{eqnarray*}
J_4%
&\le& |b| \gamma \left\|\partial_y^2 \langle u \rangle \right\|_{L_y^2} \left\|u_{\notparallel}^2 \right\|_{L^2} \le \frac{\varepsilon \gamma}{8} \left\|\partial_y^2 \langle u \rangle\right\|_{L_y^2}^2+B\gamma \int_{\T^2} u_{\notparallel}^4 dxdy \\ 
&\le&  \frac{\varepsilon \gamma}{8} \left\|\partial_y^2 \langle u \rangle\right\|_{L_y^2}^2+B\gamma \|\Delta u_{\notparallel}\|_{L^2} \|u_{\notparallel}\|_{L^2}^3,
\end{eqnarray*}
where we have used the Gagliardo–Nirenberg's inequality in 2D in the last estimate above:
$$
\|u_{\notparallel}\|_{L^4} \le B\|\Delta u_{\notparallel}\|_{L^2}^{\frac{1}{4}} \|u_{\notparallel}\|_{L^2}^{\frac{3}{4}}. 
$$

\medskip

Combining all the estimates of the terms $J_1, J_2, J_3$ and $J_4$ with \eqref{20210829eq500}, we have
\begin{eqnarray} \label{20210829eq51} 
&& \frac{1}{2} \frac{d}{dt} \|\langle u \rangle \|_{L_y^2}^2+ \frac{\varepsilon \gamma}{2} \left\|\partial_y^2 \langle u \rangle \right\|_{L_y^2}^2 \nonumber  \\
&&  \quad \quad \quad \quad  \le B\gamma K^{\frac{8}{3}} \left\|\langle u \rangle \right\|_{L_y^2}^2 + Ba^2\gamma \|\Delta u_{\notparallel}\|_{L^2}^2 \|u_{\notparallel}\|_{L^2}^2 \|\langle u \rangle\|_{L_y^2}^2 \nonumber \\
&& \quad \quad \quad \quad \quad \quad + Ba^2\gamma \|\Delta u_{\notparallel}\|_{L^2}^2 \|u_{\notparallel}\|_{L^2}^4 + B\gamma \|\Delta u_{\notparallel}\|_{L^2} \|u_{\notparallel}\|_{L^2}^3.
\end{eqnarray}
Note that by the Poincare's inequality in 1D and \eqref{20211016eq02}, we have 
$$
B\gamma K^{\frac{8}{3}} \left\| \langle u \rangle \right\|^2_{L_y^2} \le B \gamma \cdot \frac{\varepsilon \lambda_1^2}{4B} \cdot  \frac{1}{\lambda_1^2} \left\| \partial_y^2 \langle u \rangle \right\|_{L_y^2}^2=\frac{\varepsilon \gamma}{4} \left\| \partial_y^2 \langle u \rangle \right\|_{L_y^2}^2.
$$
This together with \eqref{20210829eq51} further gives
\begin{eqnarray} \label{20210829eq609}
&& \frac{1}{2} \frac{d}{dt} \|\langle u \rangle\|_{L_y^2}^2+\frac{\varepsilon \gamma}{4} \|\partial_y^2 \langle u \rangle \|_{L_y^2}^2 \le  B\gamma \big(a^2\|\Delta u_{\notparallel}\|_{L^2}^2 \|u_{\notparallel}\|_{L^2}^4 \nonumber \\
&& \quad \quad \quad \quad \quad \quad \quad  \quad + \|\Delta u_{\notparallel}\|_{L^2} \|u_{\notparallel}\|_{L^2}^3+\|\Delta u_{\notparallel}\|_{L^2}^2 \|u_{\notparallel}\|_{L^2}^2 \|\langle u \rangle\|_{L_y^2}^2 \big). 
\end{eqnarray}
In particular, we have the following ODE:
\begin{eqnarray} \label{20210829eq61}
&& \frac{1}{2} \frac{d}{dt} \|\langle u \rangle\|_{L^2}^2-B\gamma \|\Delta u_{\notparallel}\|_{L^2}^2 \|u_{\notparallel}\|_{L^2}^2 \|\langle u \rangle\|_{L_y^2}^2 \nonumber \\
&& \quad \quad \quad \quad \le  B\gamma \left(a^2\|\Delta u_{\notparallel}\|_{L^2}^2 \|u_{\notparallel}\|_{L^2}^4+ \|\Delta u_{\notparallel}\|_{L^2} \|u_{\notparallel}\|_{L^2}^3 \right). 
\end{eqnarray}
Let
$$
\rho(t):=\exp \left(-B\gamma \int_0^t \|\Delta u_{\notparallel}(\tau)\|_{L^2}^2 \|u_{\notparallel}(\tau)\|_{L^2}^2 d\tau \right)
$$
be the integrating factor, and solve the ODE \eqref{20210829eq61}, we have
\begin{eqnarray} \label{20210829eq62}
&& \left\|\langle u \rangle(t) \right\|_{L_y^2}^2 \le \frac{\|\langle u \rangle(0) \|_{L_y^2}^2}{\rho(t)} +\frac{Ba^2\gamma}{\rho(t)} \int_0^t \rho(\tau) \|\Delta u_{\notparallel}(t)\|_{L^2}^2 \|u_{\notparallel}(\tau)\|_{L^2}^4 d\tau \nonumber \\
&& \quad \quad \quad \quad \quad \quad \quad +\frac{B\gamma}{\rho(t)} \int_0^t \rho(\tau)  \|\Delta u_{\notparallel}(\tau)\|_{L^2} \|u_{\notparallel}(\tau)\|_{L^2}^3 d\tau \nonumber\\
&& \quad \quad \le \frac{\|\langle u \rangle(0) \|_{L_y^2}^2}{\rho(t)} +\frac{Ba^2\gamma}{\rho(t)} \int_0^t  \|\Delta u_{\notparallel}(t)\|_{L^2}^2 \|u_{\notparallel}(\tau)\|_{L^2}^4 d\tau \nonumber \\
&& \quad \quad \quad \quad \quad \quad \quad +\frac{B\gamma}{\rho(t)} \int_0^t  \|\Delta u_{\notparallel}(\tau)\|_{L^2} \|u_{\notparallel}(\tau)\|_{L^2}^3 d\tau. 
\end{eqnarray}
By the bootstrap assumptions, we first estimate the term $\frac{1}{\rho(t)}$ as follows:
\begin{eqnarray} \label{20210829eq63}
\frac{1}{\rho(t)}%
&=& \exp \left(B\gamma \int_0^t \|\Delta u_{\notparallel}(\tau)\|_{L^2}^2 \|u_{\notparallel}(\tau)\|_{L^2}^2 d\tau \right) \nonumber \\
&\le& \exp \left( B \cdot \varepsilon \gamma \|u_{\notparallel}(0)\|_{L^2}^2 \int_0^t \|\Delta u_{\notparallel}(\tau)\|_{L^2}^2 d\tau \right) \nonumber \\
&\le& \exp \left( B \|u_{\notparallel}(0)\|_{L^2}^4 \right). 
\end{eqnarray}
Next, we estimate the two integrals in \eqref{20210829eq62} as follows:
\begin{eqnarray} \label{20210829eq64}
 \gamma \int_0^t \|\Delta u_{\notparallel}(\tau)\|^2_{L^2} \|u_{\notparallel}(\tau)\|_{L^2}^4 d\tau%
 &\le&  B \|u_{\notparallel}(0)\|_{L^2}^4 \cdot \varepsilon \gamma \int_0^t \|\Delta u_{\notparallel}(\tau) \|_{L^2}^2 d\tau \nonumber \\
 &\le & B \|u_{\notparallel}(0)\|_{L^2}^6
\end{eqnarray}
and
\begin{eqnarray} \label{20210829eq65}
&& \gamma \int_0^t \|\Delta u_{\notparallel}(\tau) \|_{L^2} \|u_{\notparallel}(\tau)\|_{L^2}^3 d\tau \le B\|u_{\notparallel}(0)\|_{L^2}^3 \cdot \gamma \int_0^t e^{-\frac{3 \lambda_\gamma \tau}{4}} \|\Delta u_{\notparallel}(\tau) \|_{L^2} d\tau \nonumber \\
&&  \quad \quad =B \|u_{\notparallel}(0)\|^3_{L^2} \cdot  \gamma^{\frac{1}{2}} \cdot \left(\int_0^t e^{-\frac{3\lambda_\gamma \tau}{2}} d\tau \right)^{\frac{1}{2}} \cdot \left(\varepsilon \gamma \int_0^t \left\|\Delta u_{\notparallel}(\tau) \right\|_{L^2}^2 d\tau \right)^{\frac{1}{2}} \nonumber \\ 
&& \quad  \quad  \le  B  \cdot \left(\frac{\gamma}{\lambda_{\gamma}}  \right)^{\frac{1}{2}} \|u_{\notparallel}(0)\|_{L^2}^4.
\end{eqnarray}

\medskip

Therefore, by \eqref{20210829eq62}, \eqref{20210829eq63}. \eqref{20210829eq64} and \eqref{20210829eq65}, we have
\begin{eqnarray*} \label{20210829eq66}
\left\|\langle u \rangle \left(\widetilde{t_0} \right)\right\|_{L_y^2}^2 %
&\le&  B \exp \left( B\|u_{\notparallel}(0)\|_{L^2}^4 \right) \bigg( a^2\left\|u_{\notparallel}(0) \right\|_{L^2}^6 \nonumber \\
&&  \quad \quad \quad \quad + \left(\frac{\gamma}{\lambda_\gamma} \right)^{\frac{1}{2}} \left\|u_{\notparallel}(0) \right\|_{L^2}^4+\left\|\langle u \rangle(0) \right\|_{L_y^2}^2 \bigg).
\end{eqnarray*} 
The proof is complete. 
\end{proof}

\begin{prop} \label{avgprop}
Assume the bootstrap assumptions and $|a|$ is sufficiently small with satisfying
\begin{equation} \label{20211017eq04}
a^2 {\bf B_2} \exp \left({\bf B_2} \left\|u_{\notparallel}(0) \right\|_{L^2}^4 \right) \left\|u_{\notparallel}(0) \right\|_{L^2}^6 \le  \min \left\{ \frac{1}{12}, \ \frac{1}{12} \left(\frac{\varepsilon \lambda_1^2}{4{\bf B_1}} \right)^{\frac{3}{4}} \right\},
\end{equation} 
where ${\bf B_1}, {\bf B_2}>1$ are the constants defined in Lemma \ref{priopest}. Then there exists a $\gamma_0>0$, which only depends on $\varepsilon, b, \left\|u_{\notparallel}(0) \right\|_{L^2}$ and any dimensional constants , such that for any $0<\gamma<\gamma_0$ and for any initial data $\langle u \rangle (0, \cdot)$ of \eqref{20210822eq03} with satisfying 
\begin{equation} \label{20211017eq01} 
    \left\| \langle u \rangle (0, \cdot) \right\|_{L_y^2} \le \frac{1}{{\bf B_2} \exp( {\bf B_2} )}\min \left\{\frac{1}{12}, \ \frac{1}{12} \left(\frac{\varepsilon \lambda_1^2}{4{\bf B_1}} \right)^{\frac{3}{8}} \right\},
\end{equation}
the following estimates hold:  for any $0 \le t \le t_0$, 
\begin{equation} \label{20211017eq02}
    \left\| \langle u \rangle (t, \cdot) \right\|_{L_y^2}^2 + \varepsilon \gamma \int_0^t \left\|\partial_y^2 \langle u \rangle (\tau, \cdot) \right\|_{L_y^2}^2 d\tau \le {\bf B_3},
    \end{equation}
where ${\bf B_3}>0$ is a constant only depending on $\varepsilon, b, \left\| u_{\notparallel}(0) \right\|_{L^2}$ and any dimensional constants.  
\end{prop}

\begin{proof}
(1). Since $\frac{\gamma}{\lambda_\gamma} \to 0$ as $\gamma \to 0$, this allows us to choose a $\gamma_0$ sufficiently small, such that for any $0<\gamma<\gamma_0$, it holds that 
\begin{equation} \label{20211017eq11}
{\bf B_2} \exp \left( {\bf B_2} \left\|u_{\notparallel}(0) \right\|_{L^2}^4 \right) \|u_{\notparallel}(0)\|_{L^2}^4  \left(\frac{\gamma}{\lambda_\gamma} \right)^{\frac{1}{2}}<  \min \left\{ \frac{1}{12}, \ \frac{1}{12} \left(\frac{\varepsilon \lambda_1^2}{4{\bf B_1}} \right)^{\frac{3}{4}} \right\}.
\end{equation} 
Note that here $\gamma_0$ only depends on $\varepsilon, b, \left\|u_{\notparallel}(0) \right\|_{L^2}$ and any dimensional constants. We now let
$$
K:= \min \left\{\frac{1}{2}, \ \frac{1}{2} \left(\frac{\varepsilon \lambda_1^2}{4{\bf B_1}} \right)^{\frac{3}{8}} \right\},
$$
in Lemma \ref{priopest} and $t(K)$ be the maximal time such that $\left\| \langle u \rangle (t, \cdot) \right\|_{L_y^2} \le K$ on $[0, t(K)]$. Without loss of generality, we might assume $K=\frac{1}{2} \left(\frac{\varepsilon \lambda_1^2}{4{\bf B_1}} \right)^{\frac{3}{8}}$, otherwise we can take a larger ${\bf B_1}$ in Lemma \ref{priopest} to make such an assumption hold. By the continuity of $L^2$ norm of the mild solution, $t(K)>0$. Our goal is to show that $t(K)=t_0$. Otherwise, assume $0<t(K)<t_0$. By Lemma \ref{priopest}, \eqref{20211017eq04},  \eqref{20211017eq01} and \eqref{20211017eq11}, we have for any $0<\gamma<\gamma_0$, 
\begin{eqnarray*}
&& \frac{1}{4} \left(\frac{\varepsilon \lambda_1^2}{4{\bf B_1}} \right)^{\frac{3}{4}} = K^2= \left\| \langle u \rangle (t(K)) \right\|_{L_y^2}^2 \nonumber \\
&& \quad \le a^2 {\bf B_2} \exp \left( {\bf B_2} \|u_{\notparallel}(0)\|_{L^2}^4 \right)  \left\|u_{\notparallel}(0) \right\|_{L^2}^6  \\
&& \quad \quad  \quad  + {\bf B_2} \exp \left( {\bf B_2} \|u_{\notparallel}(0)\|_{L^2}^4 \right) \left[\left(\frac{\gamma}{\lambda_\gamma} \right)^{\frac{1}{2}} \left\|u_{\notparallel}(0) \right\|_{L^2}^4+\left\|\langle u \rangle(0) \right\|_{L_y^2}^2  \right] \nonumber \\
&& \quad < \left[ \frac{1}{6}+  \frac{1}{12} \cdot \frac{{\bf B_2} \exp \left({{\bf B_2}\left\|u_{\notparallel}(0) \right\|_{L^2}^4} \right)}{{\bf B_2} \exp ({\bf B_2})}  \right] \left(\frac{\varepsilon \lambda_1^2}{4{\bf B_1}} \right)^{\frac{3}{4}} \\
&& \quad \le \frac{1}{4} \left(\frac{\varepsilon \lambda_1^2}{4{\bf B_1}} \right)^{\frac{3}{4}}, 
\end{eqnarray*}
where in the last estimate, we have used the fact that $\left\| u_{\notparallel}(0) \right\|_{L^2} \le 1$, which is guaranteed by the assumption \eqref{20211017eq04}. This clearly gives a contradiction and therefore the estimate 
\begin{equation} \label{20211017eq020}
\left\| \langle u \rangle (t, \cdot) \right\|_{L_y^2} \le \frac{1}{4} \left(\frac{\varepsilon \lambda_1^2}{4{\bf B_1}} \right)^{\frac{3}{4}}, 
\end{equation} 
has to be true until the maximal time, namely, $t(K)=t_0$. This gives the first part of the estimate \eqref{20211017eq02}. 

\medskip

(2). Now we turn to prove the second part of the estimate \eqref{20211017eq02}.  Taking the time integral on both sides of \eqref{20210829eq609} and using \eqref{20211017eq11},  \eqref{20211017eq020} and the bootstrap assumptions, we find that for any $0 \le t \le t_0$ and $0<\gamma<\gamma_0$,
\begin{eqnarray*}
\varepsilon \gamma \int_0^t \left\|\partial_y^2 \langle u \rangle(\tau) \right\|_{L_y^2}^2 d\tau%
&\le& B\gamma \int_0^t \left\|\Delta u_{\notparallel}(\tau) \right\|_{L^2}^2 \|u_{\notparallel}(\tau)\|_{L^2}^4  d\tau \\
&& +B\gamma \int_0^t \|\Delta u_{\notparallel}(\tau)\|_{L^2}^2  \|u_{\notparallel}(\tau)\|_{L^2}^2 \|\langle u \rangle(\tau)\|_{L_y^2}^2 d\tau \\
&& + B\gamma \int_0^t \|\Delta u_{\notparallel}(\tau)\|_{L^2} \|u_{\notparallel}(\tau)\|_{L^2}^3 d\tau \le {\bf B_3}, 
\end{eqnarray*}
where ${\bf B_3}$ is a constant which only depends on $\varepsilon, b$, $\left\|u_{\notparallel}(0)\right\|_{L^2}$ and any dimensional constants . The proof is complete. 
\end{proof}

\bigskip

\section{Bootstrap estimates and Proof of Theorem \ref{mainthm}} \label{bootest} 

In this section, we show that the bootstrap assumptions can be improved for $\gamma$ being sufficiently small, that is, for example, for $0 \le s \le t \le t_0$, 
\begin{enumerate}
    \item [(1).] $\|u_{\notparallel}(t)\|_{L^2} \le 15 e^{-\frac{\lambda_{\gamma}(t-s)}{4}} \|u_{\notparallel}(s)\|_{L^2}$;
    \item [(2).] $\varepsilon \gamma \int_s^t \|\Delta u_{\notparallel}(\tau)\|_{L^2}^2 d\tau  \le 5 \|u_{\notparallel}(s)\|_{L^2}^2$. 
\end{enumerate}
Note that this then gives a contradiction if $t_0<\infty$, and this allows us to conclude the global existence of the solution to the problem \eqref{maineq02}. Now we turn to some details. 

The following result shows that the second estimate in the bootstrap assumptions can be improved. 

\begin{prop} \label{bootestprop1}
Assume the bootstrap assumptions and \eqref{20211017eq04}. Moreover, we assume that  
\begin{equation} \label{20211018eq01}
|a| \le \frac{\varepsilon}{10^7{\bf L} \left\|u_{\notparallel}(0) \right\|^2_{L^2}}, 
\end{equation}
where ${\bf L}>0$ is some dimensional constant sufficiently large, then there exists a $0<\gamma_1 \le \gamma_0$ sufficiently small, which only depends on $\varepsilon, a, b,  \|\langle u \rangle(0)\|_{L^2}, \left\| u_{\notparallel}(0) \right\|_{L^2}$ and any dimensional constants, such that for any $0<\gamma<\gamma_1$ and any $0 \le s \le t \le t_0$,
\begin{equation} \label{20210904eq02}
\varepsilon \gamma \int_s^t \left\|\Delta u_{\notparallel}(\tau) \right\|_{L^2}^2 d\tau \le 5 \left\|u_{\notparallel}(s) \right\|_{L^2}^2. 
\end{equation} 
Here, $\gamma_0$ is defined in Proposition \ref{avgprop}. 
\end{prop}

\begin{proof}
Multiplying $u_{\notparallel}$ on both sides of \eqref{20210822eq04}, we have 
\begin{eqnarray} \label{20210902eq02}
 \frac{1}{2} \frac{d}{dt} \|u_{\notparallel}\|_{L^2}^2+\varepsilon \gamma \left\|\Delta u_{\notparallel}\right\|_{L^2}^2%
 &=& a\gamma \int_{\T^2} u_{\notparallel} \Delta \left[u_{\notparallel}^3+3 \langle u \rangle^2 u_{\notparallel}+3 \langle u \rangle u_{\notparallel}^2 \right] dxdy \nonumber \\
&&  +b\gamma \int_{\T^2} u_{\notparallel} \Delta \left[ u_{\notparallel}^2+2 \langle u \rangle u_{\notparallel} \right]dxdy \nonumber \\
&=& K_1+K_2+K_3+K_4, 
\end{eqnarray}
where 
$$
K_1:=\gamma \int_{\T^2} \Delta u_{\notparallel} \left(au_{\notparallel}^3+bu_{\notparallel}^2 \right) dxdy,  \quad K_2:=3a\gamma \int_{\T^2} \Delta u_{\notparallel} \langle u \rangle^2 u_{\notparallel} dxdy
$$
$$
K_3:=3a\gamma \int_{\T^2} \Delta u_{\notparallel} \langle u \rangle u_{\notparallel}^2 dxdy \quad \textrm{and} \quad K_4:=2b \gamma \int_{\T^2} \Delta u_{\notparallel} \langle u \rangle u_{\notparallel} dxdy. 
$$

\medskip

\textit{Estimate of $K_1$.} By the Young's inequality and the Gagliardo-Nirenberg's inequalities in 2D:
\begin{equation} \label{20210902eq11}
\left\|u_{\notparallel} \right\|_{L^4} \le B \|\Delta u_{\notparallel} \|_{L^2}^{\frac{1}{4}} \|u_{\notparallel}\|_{L^2}^{\frac{3}{4}}, 
\end{equation} 
and
$$
\left\| \nabla u_{\notparallel} \right\|_{L^4} \le B \left\| \Delta u_{\notparallel} \right\|_{L^2}^{\frac{3}{4}} \left\|u_{\notparallel} \right\|_{L^2}^{\frac{1}{4}},
$$
we have
\begin{eqnarray*}
K_1%
&\le& 3|a| \gamma \int_{\T^2} \left\|u_{\notparallel} \right|^2 \left| \nabla u_{\notparallel} \right|^2 dxdy+2|b| \gamma \int_{\T} \left| \nabla u_{\notparallel} \right|^2 \left\|u_{\notparallel} \right| dxdy \\\
&\le& 3|a|\gamma \left\|u_{\notparallel} \right\|_{L^4}^2 \left\|\nabla u_{\notparallel} \right\|_{L^4}^2+2|b| \gamma \left\|\nabla u_{\notparallel} \right\|_{L^4}^2 \left\|u_{\notparallel} \right\|_{L^2} \\
&\le& {\bf L} |a| \gamma \left\|\Delta u_{\notparallel} \right\|_{L^2}^2 \left\| u_{\notparallel} \right\|_{L^2}^2+B\gamma \left\|\Delta u_{\notparallel} \right\|_{L^2}^{\frac{3}{2}} \left\|u_{\notparallel} \right\|_{L^2}^{\frac{3}{2}},
\end{eqnarray*}
where ${\bf L}>0$ is some dimensional constant.

\medskip

\textit{Estimate of $K_2$.} By Young's inequality and the Gagliardo–Nirenberg's inequality in 1D:
\begin{equation} \label{20210902eq10}
\|\langle u \rangle\|_{L_y^\infty} \le B\left\|\partial_y^2 \langle u \rangle \right\|_{L_y^2}^\frac{1}{4} \left\|\langle u \rangle \right\|_{L_y^2}^{\frac{3}{4}}, 
\end{equation} 
 we have
\begin{eqnarray*}
K_2%
&\le& 3a\gamma \left\|\Delta u_{\notparallel} \right\|_{L^2}  \left( \int_{\T^2} \langle u \rangle^4 u_{\notparallel}^2 dxdy \right)^{\frac{1}{2}}\le  \frac{\varepsilon \gamma}{10} \left\|\Delta u_{\notparallel}\right\|_{L^2}^2+Ba^2\gamma \int_{\T^2} \langle u \rangle^4 u_{\notparallel}^2 dxdy \\
&\le& \frac{\varepsilon  \gamma}{10} \|\Delta u_{\notparallel}\|_{L^2}^2+Ba^2\gamma \left\|\langle u \rangle \right\|_{L_y^\infty}^4 \left\|u_{\notparallel} \right\|_{L^2}^2 \\
&\le&  \frac{\varepsilon  \gamma}{10} \left\| \Delta u_{\notparallel} \right\|_{L^2}^2+B\gamma \left\|\partial_y^2 \langle u \rangle \right\|_{L_y^2} \left\|\langle u \rangle \right\|_{L_y^2}^3 \left\|u_{\notparallel}\right\|_{L^2}^2,
\end{eqnarray*}
where in the last estimate, we have used the assumption that $|a|<1$.

\medskip

\textit{Estimate of $K_3$.} By Young's inequality, \eqref{20210902eq11} and \eqref{20210902eq10}, we have
\begin{eqnarray*}
K_3%
&\le& 3a\gamma \left\|\Delta u_{\notparallel} \right\|_{L^2} \left( \int_{\T^2} \langle u \rangle^2 u_{\notparallel}^4 dxdy \right)^{\frac{1}{2}} \le \frac{\varepsilon \gamma}{10} \left\|\Delta u_{\notparallel} \right\|_{L^2}^2+B\gamma \int_{\T^2} \langle u \rangle^2 u_{\notparallel}^4 dxdy \\
&\le& \frac{\varepsilon  \gamma}{10} \left\|\Delta u_{\notparallel} \right\|_{L^2}^2+B \gamma \left\|\langle u \rangle \right\|_{L_y^\infty}^2 \int_{\T^2} u_{\notparallel}^4 dxdy \\
&\le& \frac{\varepsilon  \gamma}{10} \left\|\Delta u_{\notparallel} \right\|_{L^2}^2+B\gamma \left\|\partial_y^2 \langle u \rangle \right\|_{L_y^2}^{\frac{1}{2}} \left\| \langle u \rangle \right\|_{L_y^2} \left\|\Delta u_{\notparallel} \right\|_{L^2} \left\|u_{\notparallel} \right\|_{L^2}^3.
\end{eqnarray*}

\medskip

\textit{Estimate of $K_4$.} Again, by Young's inequality and the Gagliardo–Nirenberg's inequalities \eqref{20210902eq10} and \eqref{20210902eq11}, we see that
\begin{eqnarray*}
K_4%
&\le& 2b\gamma \left\|\Delta u_{\notparallel} \right\|_{L^2} \left( \int_{\T^2} \langle u \rangle^2 u_{\notparallel}^2 dxdy \right)^{\frac{1}{2}} \le \frac{\varepsilon  \gamma}{10} \left\|\Delta u_{\notparallel} \right\|_{L^2}^2+B\gamma \int_{\T^2} \langle u \rangle^2 u_{\notparallel}^2 dxdy \\
&\le& \frac{\varepsilon  \gamma}{10} \left\|\Delta u_{\notparallel} \right\|_{L^2}^2+B\gamma \left\| \langle u \rangle \right\|_{L_y^\infty}^2 \left\|u_{\notparallel} \right\|_{L^2}^2 \\
&\le& \frac{\varepsilon  \gamma}{10} \left\|\Delta u \right\|_{L^2}^2+B\gamma \left\|\partial_y^2 \langle u \rangle \right\|_{L_y^2}^{\frac{1}{2}} \left\|\langle y \rangle \right\|_{L_y^2}^{\frac{3}{2}} \left\|u_{\notparallel} \right\|_{L^2}^2. 
\end{eqnarray*}

\medskip

Combining \eqref{20210902eq02} with estimates of the terms $K_1, K_2, K_3$ and $K_4$, we derive that
\begin{eqnarray} \label{20210902eq04} 
&& \frac{1}{2} \frac{d}{dt} \left\|u_{\notparallel} \right\|_{L^2}^2+\frac{\varepsilon  \gamma}{4} \left\|\Delta u_{\notparallel} \right\|_{L^2}^2 \le {\bf L} |a| \gamma \left\|\Delta u_{\notparallel} \right\|_{L^2}^2 \left\|u_{\notparallel} \right\|_{L^2}^2 \nonumber \\
&& \quad \quad \quad B\gamma \bigg[ \left\|\Delta u_{\notparallel} \right\|_{L^2}^{\frac{3}{2}} \left\|u_{\notparallel} \right\|_{L^2}^{\frac{3}{2}}+ \left\|\partial_y^2 \langle u \rangle \right\|_{L_y^2} \left\|\langle u \rangle \right\|_{L_y^2}^3 \left\|u_{\notparallel}\right\|_{L^2}^2 \nonumber  \\
&& \quad \quad \quad    +  \left\|\partial_y^2 \langle u \rangle \right\|_{L_y^2}^{\frac{1}{2}} \left\| \langle u \rangle \right\|_{L_y^2} \left\|\Delta u_{\notparallel} \right\|_{L^2} \left\|u_{\notparallel} \right\|_{L^2}^3+  \left\|\partial_y^2 \langle u \rangle \right\|_{L_y^2}^{\frac{1}{2}} \left\|\langle y \rangle \right\|_{L_y^2}^{\frac{3}{2}} \left\|u_{\notparallel} \right\|_{L^2}^2 \bigg],  
\end{eqnarray}
which implies 
\begin{equation} \label{20210902eq03}
\varepsilon \gamma \int_s^t \left\|\Delta u_{\notparallel} (\tau) \right\|_{L^2}^2 d\tau \le 2\left\|u_{\notparallel}(s) \right\|_{L^2}^2+\widetilde{K_0}+\widetilde{K_1}+\widetilde{K_2}+\widetilde{K_3}+\widetilde{K_4},
\end{equation}
where 
$$
\widetilde{K_0}:=4{\bf L} |a| \gamma \int_s^t \left\|\Delta u_{\notparallel}(\tau) \right\|_{L^2}^2 \left\|u_{\notparallel}(\tau) \right\|_{L^2}^2 d\tau,
$$
$$
\widetilde{K_1}:=L\gamma \int_s^t \left\|\Delta u_{\notparallel}(\tau) \right\|_{L^2}^{\frac{3}{2}} \left\|u_{\notparallel}(\tau) \right\|_{L^2}^{\frac{3}{2}} d\tau, 
$$
$$
\widetilde{K_2}:=B\gamma \int_s^t \left\|\partial_y^2 \langle u \rangle (\tau) \right\|_{L_y^2} \left\|\langle u \rangle (\tau) \right\|_{L_y^2}^3 \left\|u_{\notparallel} (\tau) \right\|_{L^2}^2 d\tau, 
$$
$$
\widetilde{K_3}:=B\gamma \int_s^t \left\|\partial_y^2 \langle u \rangle \right\|_{L_y^2}^{\frac{1}{2}} \left\| \langle u \rangle \right\|_{L_y^2} \left\|\Delta u_{\notparallel} \right\|_{L^2} \left\|u_{\notparallel} \right\|_{L^2}^3 d\tau
$$
and
$$
\widetilde{K_4}:=B\gamma \int_s^t \left\|\partial_y^2 \langle u \rangle (\tau) \right\|_{L_y^2}^{\frac{1}{2}} \left\|\langle u \rangle(\tau) \right\|_{L_y^2}^{\frac{3}{2}} \left\|u_{\notparallel} (\tau) \right\|_{L^2}^2 d\tau. 
$$

\medskip

\textit{Estimate of $\widetilde{K_0}$.} By the bootstrap assumptions and \eqref{20211018eq01}, we have
\begin{eqnarray*}
\widetilde{K_0}%
&\le& \frac{1600{\bf L}|a|}{\varepsilon} \|u_{\notparallel}(0)\|_{L^2}^2 \cdot \varepsilon \gamma \int_s^t \left\|\Delta u_{\notparallel}(\tau) \right\|_{L^2}^2 d\tau \\
&\le& \frac{16000 {\bf L} |a|}{\varepsilon} \|u_{\notparallel}(0)\|_{L^2}^2 \cdot \|u_{\notparallel}(s) \|_{L^2}^2 \\
&\le& \frac{\|u_{\notparallel}(s)\|_{L^2}^2}{20}.
\end{eqnarray*}

\textit{Estimate of $\widetilde{K_1}$.} By the bootstrap assumptions, we have 
\begin{eqnarray*}
\widetilde{K_1} %
&\le& B \gamma \int_s^t e^{-\frac{3 \lambda_\gamma (t-s)}{8}} \left\|\Delta u_{\notparallel}(\tau) \right\|_{L^2}^{\frac{3}{2}} d\tau \cdot \left\|u_{\notparallel}(s)\right\|_{L^2}^{\frac{3}{2}} \\ 
&\le& B\gamma \left\|u_{\notparallel}(s) \right\|_{L^2}^{\frac{3}{2}}  \left(\int_s^t e^{-\frac{3 \lambda_\gamma(\tau-s)}{2}} d\tau \right)^{\frac{1}{4}} \left(\int_s^t \left\|\Delta u_{\notparallel}(\tau) \right\|_{L^2}^2 d\tau \right)^{\frac{3}{4}} \\
&\le& B \cdot \left(\frac{\gamma}{\lambda_\gamma} \right)^{\frac{1}{4}} \cdot \left(\varepsilon \gamma \int_s^t \left\|\Delta u_{\notparallel}(\tau) \right\|_{L^2}^2 d\tau \right)^{\frac{3}{4}} \cdot \left\|u_{\notparallel}(s) \right\|_{L^2}^{\frac{3}{2}} \\
&\le& B \left(\frac{\gamma}{\lambda_\gamma} \right)^{\frac{1}{4}} \cdot \|u_{\notparallel}(s) \|_{L^2}^3 \\
&\le& B \|u_{\notparallel}(0)\|_{L^2} \left(\frac{\gamma}{ \lambda_\gamma} \right)^{\frac{1}{4}}  \cdot \left\|u_{\notparallel}(s) \right\|_{L^2}^2, 
\end{eqnarray*} 
where in the last estimate, we have used the estimate $\|u_{\notparallel}(s)\|_{L^2} \le 20 \|u_{\notparallel}(0)\|_{L^2}$, which is a consequence of the first estimate in the bootstrap assumptions.

\textit{Estimate of $\widetilde{K_2}$.}  By the bootstrap assumptions and Proposition \ref{avgprop}, we have 
\begin{eqnarray*}
\widetilde{K_2}%
&\le& C\gamma \int_s^t \left\|\partial_y^2 \langle u \rangle(\tau) \right\|_{L_y^2} \left\|u_{\notparallel} (\tau)\right\|_{L^2}^2 d\tau \\
&\le& C \gamma \left\|u_{\notparallel}(s) \right\|_{L^2}^2 \int_s^t e^{-\frac{\lambda_\gamma(\tau-s)}{2}} \left\|\partial_y^2 \langle u \rangle(\tau) \right\|_{L_y^2} d\tau \\
&\le& C\gamma \left\|u_{\notparallel}(s) \right\|_{L^2}^2 \left(\int_s^t e^{-\lambda_\gamma(\tau-s)} d\tau \right)^{\frac{1}{2}} \left(\int_s^t \left\|\partial_y^2 \langle u \rangle (\tau) \right\|_{L_y^2}^2 d\tau \right)^{\frac{1}{2}} \\
&\le& C \left(\frac{\gamma}{\lambda_{\gamma}} \right)^{\frac{1}{2}} \cdot \left\|u_{\notparallel}(s) \right\|_{L^2}^2.
\end{eqnarray*}
Note that since we have used Proposition \ref{avgprop}, the implicit constants in the above estimates can also depend on $a$, $\left\| \langle u\rangle (0) \right\|_{L^2}$ and $\left\| u_{\notparallel}(0) \right\|_{L^2}$. 

\medskip

\textit{Estimate of $\widetilde{K_3}$.} Again, using the bootstrap assumptions and Proposition \ref{avgprop}, we see that
\begin{eqnarray*}
\widetilde{K_3}%
&\le& C\gamma \left\|u_{\notparallel}(s)\right\|_{L^2}^3 \int_s^t e^{-\frac{3\lambda_{\gamma}(\tau-s)}{4}} \left\|\partial_y^2 \langle u \rangle (\tau) \right\|_{L_y^2}^{\frac{1}{2}} \left\|\Delta u_{\notparallel}(\tau) \right\|_{L^2} d\tau \\
&\le& C\gamma  \left\|u_{\notparallel}(s)\right\|_{L^2}^3  \left(\int_s^t e^{-3\lambda_{\gamma}(\tau-s)} d\tau \right)^{\frac{1}{4}}   \left(\int_s^t \left\|\partial_y^2 \langle u \rangle \right\|_{L_y^2}^2 d\tau \right)^{\frac{1}{4}} \\
&& \quad \quad \quad \cdot \left( \int_s^t \left\|\Delta u_{\notparallel}(\tau) \right\|_{L^2}^2 d\tau \right)^{\frac{1}{2}} \\
&\le& C \left(\frac{\gamma}{ \lambda_{\gamma}} \right)^{\frac{1}{4}} \left\|u_{\notparallel}(s) \right\|_{L^2}^3  \left(\varepsilon \gamma \int_s^t \left\|\partial_y^2 \langle u \rangle \right\|_{L_y^2}^2 d\tau \right)^{\frac{1}{4}} \left( \varepsilon \gamma \int_s^t \left\|\Delta u_{\notparallel}(\tau) \right\|_{L^2}^2 d\tau \right)^{\frac{1}{2}} \\
&\le& C \left(\frac{\gamma}{\lambda_{\gamma}} \right)^{\frac{1}{4}} \left\|u_{\notparallel}(s) \right\|_{L^2}^4 \le   C  \left\|u_{\notparallel}(0)\right\|_{L^2}^2 \left(\frac{\gamma}{\lambda_{\gamma}} \right)^{\frac{1}{4}} \left\|u_{\notparallel}(s)\right\|_{L^2}^2.
\end{eqnarray*}

\medskip

\textit{Estimate of $\widetilde{K_4}$.} Similarly, we have
\begin{eqnarray*}
\widetilde{K_4}%
&\le& C\gamma \int_s^t \left\|\partial_y^2 \langle u \rangle (\tau) \right\|_{L_y^2}^{\frac{1}{2}} \left\|u_{\notparallel}(\tau) \right\|_{L^2}^2 d\tau \\
&\le& C\gamma \left\|u_{\notparallel}(s) \right\|_{L^2}^2 \int_s^t e^{-\frac{\lambda_{\gamma}(\tau-s)}{2}}\left\|\partial_y^2 \langle u \rangle (\tau) \right\|_{L_y^2}^{\frac{1}{2}} d\tau \\
&\le& C\gamma^{\frac{3}{4}} \left\|u_{\notparallel}(s) \right\|_{L^2}^2  \left(\int_s^t e^{-\frac{2\lambda_{\gamma}(\tau-s)}{3}} d\tau \right)^{\frac{3}{4}}  \left(\gamma \int_s^t \left\|\partial_y^2 \langle u \rangle (\tau) \right\|_{L_y^2}^2 d\tau \right)^{\frac{1}{4}} \\
&\le& C \left(\frac{\gamma}{ \lambda_{\gamma}} \right)^{\frac{3}{4}} \left\|u_{\notparallel}(s) \right\|_{L^2}^2.
 \end{eqnarray*}
 
 \medskip

Therefore, by \eqref{20210902eq03} and the estimates of the terms $\widetilde{K_0}$, $\widetilde{K_1}$, $\widetilde{K_2}, \widetilde{K_3}$ and $\widetilde{K_4}$, we have
\begin{eqnarray} \label{20210904eq01} 
&& \varepsilon \gamma \int_s^t \left\|\Delta u_{\notparallel}(\tau) \right\|_{L^2}^2 d\tau \le 3\|u_{\notparallel}(s)\|_{L^2}^2\nonumber \\
&& \quad +C \|u_{\notparallel}(s)\|_{L^2}^2 \left[\left(\frac{\gamma}{ \lambda_{\gamma}} \right)^{\frac{3}{4}}+\left(\frac{\gamma}{\lambda_{\gamma}} \right)^{\frac{1}{2}}+ \|u_{\notparallel}(0)\|_{L^2}^2 \left(\frac{\gamma}{\lambda_{\gamma}} \right)^{\frac{1}{4}}\right].
\end{eqnarray}
Note that $\frac{\gamma}{\lambda_\gamma} \to 0$ as $\lambda \to 0$, it suffices to pick $\gamma_1$ sufficiently small, such that 
$$
C  \left[\left(\frac{\gamma}{ \lambda_{\gamma}} \right)^{\frac{3}{4}}+\left(\frac{\gamma}{\lambda_{\gamma}} \right)^{\frac{1}{2}}+ \|u_{\notparallel}(0)\|_{L^2}^2 \left(\frac{\gamma}{ \lambda_{\gamma}} \right)^{\frac{1}{4}}\right]<2, 
$$
and this together with \eqref{20210904eq01}  clearly implies the desired estimate \eqref{20210904eq02}. 
\end{proof}

Our next goal is to improve the first estimate in the bootstrap assumption. We derive this via several steps. We start with estimating $\|u_{\notparallel}(t)\|_{L^2}$, where $t$ is sufficiently close $s$.

\begin{prop} \label{bootestprop2}
Under the assumption of Proposition \ref{bootestprop1}, there exists a $0<\gamma_2 \le \gamma_0$ which only depends on $\varepsilon, a, b, \left\|\langle u \rangle(0)\right\|_{L^2}, \left\|u_{\notparallel}(0) \right\|_{L^2}$ and any dimensional constants, such that for any $0<\gamma<\gamma_2$ and any $0\le s \le t \le t_0$, 
\begin{equation} \label{20210904eq04}
\left\|u_{\notparallel}(t) \right\|_{L^2} \le \frac{3}{2} \left\| u_{\notparallel}(s) \right\|_{L^2}.    
\end{equation}
\end{prop}

\begin{proof}
Note that by \eqref{20210902eq04}, it is immediate that
\begin{eqnarray*}
\left\|u_{\notparallel}(t) \right\|_{L^2}^2 %
&\le&   \left\|u_{\notparallel}(s) \right\|_{L^2}^2+\frac{\widetilde{K_0}+\widetilde{K_1}+\widetilde{K_2}+\widetilde{K_3}+\widetilde{K_4}}{2} \\
&\le& 1.1 \left\|u_{\notparallel}(s) \right\|_{L^2}^2+ \frac{\widetilde{K_1}+\widetilde{K_2}+\widetilde{K_3}+\widetilde{K_4}}{2}
\end{eqnarray*}
where $\widetilde{K_0}, \widetilde{K_1}, \widetilde{K_2}, \widetilde{K_3}$ and $\widetilde{K_4}$ are defined in the proof of Proposition \ref{bootestprop1} (recall that $\widetilde{K_0} \le \frac{\|u_{\notparallel}(s)\|_{L^2}^2}{10}$. The desired estimate \eqref{20210904eq04} then follows from the argument that we have used in \eqref{20210902eq03}. 
\end{proof}

Next, we estimate $\|u_{\notparallel}(t)\|_{L^2}$, where $t$ is relatively ``far away" from $s$. We start with rewriting the Duhamel's formula \eqref{20210823eq01} into a slightly different form: for any $\widetilde{\tau}>0$ with $0 \le s \le s+\widetilde{\tau} \le t_0$, we can write
\begin{eqnarray} \label{20210905eq01}
u_{\notparallel}(s+\widetilde{\tau})%
&=& \calS_{\widetilde{\tau}} \left(u_{\notparallel}(s) \right)  +a\gamma \int_s^{\widetilde{\tau}+s} \calS_{\widetilde{\tau}+s-\tau} \bigg[ \Delta \left(u_{\notparallel}^3+3\langle u \rangle^2 u_{\notparallel}+3 \langle u \rangle u_{\notparallel}^2 \right) \nonumber \\
&& \quad \quad \quad \quad \quad \quad \quad \quad \quad  \quad -\int_{\T}\Delta \left(u_{\notparallel}^3+3\langle u \rangle^2 u_{\notparallel}+3 \langle u \rangle u_{\notparallel}^2 \right) dx \bigg] d\tau \nonumber \\
&& +b\gamma \int_s^{\widetilde{\tau}+s} \calS_{\widetilde{\tau}+s-\tau} \left[ \Delta\left(u_{\notparallel}^2+2\langle u \rangle u_{\notparallel} \right)-\int_{\T}\Delta\left(u_{\notparallel}^2+2\langle u \rangle u_{\notparallel} \right)  \right] d\tau. 
\end{eqnarray}

\begin{prop} \label{bootestprop3}
Assume the bootstrap assumptions, \eqref{20211017eq04} , \eqref{20211018eq01} and 
\begin{equation} \label{20211026eq100}
|a| \le \frac{\varepsilon}{10^6 {\bf L'}{\bf B_3}^{\frac{1}{2}} \|u_{\notparallel}(0)\|_{L^2}},
\end{equation}
where ${\bf B_3}$ is the constant defined in Proposition \ref{avgprop} and ${\bf L'>0}$ is some dimensional constant. Let $\tau^*:=\frac{4}{\lambda_{\gamma}}$. If $t_0 \ge \tau^*$,  then there exists a $0<\gamma_3 \le \gamma_0$, which only depends on $\varepsilon, a, b, \left\| \langle u \rangle(0) \right\|_{L^2}$ and $\left\| u_{\notparallel}(0) \right\|_{L^2}$, such that for any $0<\gamma<\gamma_3$, one has
\begin{equation} \label{20210905eq02}
\left\|u_{\notparallel}(\tau^*+s) \right\|_{L^2} \le \frac{1}{e} \left\|u_{\notparallel}(s) \right\|_{L^2}.
\end{equation}
\end{prop}

\begin{proof} 
Taking $L^2$ norm on both sides of \eqref{20210905eq01}, we have
\begin{equation} \label{20210908eq10}
\left\|u_{\notparallel}(\tau^*+s) \right\|_{L^2} \le H_1+H_2+H_3, 
\end{equation} 
where
$$
H_1:=\left\|\calS_{\tau^*}(u_{\notparallel}(s)) \right\|_{L^2}, \quad 
H_2:=2|a| \gamma \int_s^{\tau^*+s} \left\|\Delta \left(u_{\notparallel}^3+3\langle u \rangle^2 u_{\notparallel}+3\langle u\rangle u_{\notparallel}^2 \right) \right\|_{L^2} d\tau 
$$
and
$$
H_3:=2|b|\gamma \int_s^{\tau^*+s} \left\|\Delta \left(u_{\notparallel}^2+2\langle u \rangle u_{\notparallel} \right) \right\|_{L^2} d\tau. 
$$

\medskip

\textit{Estimate of $H_1$.} By Proposition \ref{20210823prop01}, we have
$$
H_1 \le 10e^{-\lambda_\gamma \tau^*} \|u_{\notparallel}(s)\|_{L^2}=  \frac{10}{e^4} \|u_{\notparallel}(s)\|_{L^2}. 
$$

\medskip

The estimates of $H_2$ and $H_3$ are the most technical parts of this paper. For simplicity, we would like to first collect all the Gagliardo–Nirenberg's inequalities that one might need for these estimates. 

\medskip

$\bullet$ Gagliardo–Nirenberg's inequalities in 1D:
\begin{equation} \label{GN1-1}
\|\langle u \rangle\|_{L_y^\infty} \le C \left\| \partial_y^2 \langle u \rangle \right\|_{L_y^2}^{\frac{1}{4}} \left\| \langle u \rangle \right\|_{L_y^2}^{\frac{3}{4}}, \tag{$G_{1, 1}$} 
\end{equation}
and
\begin{equation} \label{GN1-2}
\left\|\partial_y \langle u \rangle \right\|_{L_y^4} \le C \left\|\partial_y^2 \langle u \rangle \right\|_{L_y^2}^{\frac{5}{8}} \left\| \langle u \rangle \right\|_{L_y^2}^{\frac{3}{8}}. \tag{$G_{1, 2}$} 
\end{equation}

$\bullet$ Gagliardo–Nirenberg's inequalities in 2D:
\begin{equation} \label{GN2-1} 
\left\|u_{\notparallel} \right\|_{L^\infty} \le C \left\|\Delta u_{\notparallel} \right\|_{L^2}^{\frac{1}{2}} \|u_{\notparallel}\|_{L^2}^{\frac{1}{2}} \tag{$G_{2, 1}$} 
\end{equation}
and
\begin{equation} \label{GN2-2} 
\left\| \nabla u_{\notparallel} \right\|_{L^4} \le C \left\|\Delta u_{\notparallel} \right\|_{L^2}^{\frac{3}{4}} \left\|u_{\notparallel} \right\|_{L^2}^{\frac{1}{4}}. \tag{$G_{2, 2}$} 
\end{equation}

\medskip

\textit{Estimate of $H_2$.}  By triangle inequality, we further bound $H_2$ as follows:
\begin{equation} \label{20210909eq01}
H_2 \le H_{2, 1}+H_{2, 2}+H_{2, 3},
\end{equation} 
where 
$$
H_{2, 1}:=2|a|\gamma \int_s^{\tau^*+s} \left\|\Delta \left(u_{\notparallel}^3 \right) \right\|_{L^2} d\tau,
\quad H_{2, 2}:=6|a|\gamma \int_s^{\tau^*+s} \left\|\Delta \left( \langle u \rangle^2 u_{\notparallel} \right) \right\|_{L^2} d\tau
$$
and
$$
H_{2, 3}:=6|a|\gamma \int_s^{\tau^*+s} \left\|\Delta \left(\langle u \rangle u_{\notparallel}^2 \right) \right\|_{L^2} d\tau.
$$

\medskip

\textbf{Estimate of $H_{2, 1}$.} Note that 
$$
\Delta(u_{\notparallel}^3)=3u_{\notparallel}^2 \Delta u_{\notparallel}+6u_{\notparallel} \left|\nabla u_{\notparallel} \right|^2.
$$
Therefore, by the Gagliardo–Nirenberg's inequalities \eqref{GN2-1} and \eqref{GN2-2}, we have
\begin{eqnarray} \label{20210909eq21}
H_{2, 1}%
&=& 2|a|\gamma \int_s^{\tau^*+s} \left\|3u_{\notparallel}^2 \Delta u_{\notparallel}+6u_{\notparallel} \left| \nabla u_{\notparallel} \right|^2 \right\|_{L^2} d\tau \nonumber \\
&\le& 6|a|\gamma \int_s^{\tau^*+s} \left\|u_{\notparallel}  \right\|_{L^\infty}^2 \left\|\Delta u_{\notparallel} \right\|_{L^2} d\tau+12|a| \gamma \int_s^{\tau^*+s} \left\|u_{\notparallel} \right\|_{L^\infty} \left\|\nabla u_{\notparallel} \right\|_{L^4}^2 d\tau  \nonumber \\
&\le& {\bf L} |a| \gamma \int_s^{\tau^*+s} \left\|\Delta u_{\notparallel} \right\|_{L^2}^2 \left\|u_{\notparallel} \right\|_{L^2} d\tau \nonumber \\
&\le&  \frac{20{\bf L}|a| \|u_{\notparallel}(s) \|_{L^2}}{\varepsilon} \cdot \left(\varepsilon  \gamma \int_s^{\tau^*+s} \left\|\Delta u_{\notparallel}(\tau) \right\|_{L^2}^2 d\tau \right) \nonumber \\
&\le& \frac{200 {\bf L}|a|}{\varepsilon} \cdot \|u_{\notparallel}(s) \|_{L^2}^3 \le \frac{80000 {\bf L}|a| \|u_{\notparallel}(0)\|_{L^2}^2}{\varepsilon} \cdot \|u_{\notparallel}(s)\|_{L^2} \le \frac{\|u_{\notparallel}(s)\|_{L^2}}{20},
\end{eqnarray}
where in the second last estimate above, we have used the assumption \eqref{20211018eq01} for some dimensional constant ${\bf L}>0$ and the fact that $\|u_{\notparallel}(s)\|_{L^2} \le 20 \|u_{\notparallel}(0)\|_{L^2}$, which is a consequence of the first estimate in the bootstrap assumptions. 

\medskip

\textbf{Estimate of $H_{2, 2}$.} A simple computation gives 
$$
\Delta \left( \langle u \rangle^2 u_{\notparallel} \right)=\langle u \rangle^2 \Delta u_{\notparallel}+2 \left|\partial_y \langle u\rangle \right|^2 u_{\notparallel}+2\langle u \rangle \partial_y^2 \langle u \rangle u_{\notparallel}+4 \langle u\rangle \partial_y \langle u\rangle \partial_y u_{\notparallel}. 
$$
Therefore, by the Gagliardo–Nirenberg's inequalities \eqref{GN1-1}, \eqref{GN1-2}, \eqref{GN2-1} and \eqref{GN2-2}, we have
\begin{eqnarray} \label{20210909eq20} 
H_{2, 2}%
& \le & 6|a|\gamma \int_s^{\tau^*+s} \left\|\langle u \rangle^2 \Delta u_{\notparallel} \right\|_{L^2} d\tau+ 12|a|\gamma \int_s^{\tau^*+s} \left\| \left(\partial_y \langle u\rangle \right)^2 u_{\notparallel} \right\|_{L^2} d\tau \nonumber \\
&& +12|a| \gamma \int_s^{\tau^*+s} \left\| \langle u \rangle \partial_y^2 \langle u \rangle u_{\notparallel} \right\|_{L^2} d\tau+24|a| \gamma \int_s^{\tau^*+s} \left\|\langle u \rangle \partial_y \langle u \rangle \partial_y u_{\notparallel} \right\|_{L^2} d\tau \nonumber\\
&\le& 6|a|\gamma \int_s^{\tau^*+s} \left\|\langle u\rangle \right\|_{L_y^\infty}^2 \left\|\Delta u_{\notparallel} \right\|_{L^2} d\tau+12|a| \gamma \int_s^{\tau^*+s} \left\|u_{\notparallel} \right\|_{L^\infty} \left\|\partial_y \langle u \rangle \right\|_{L_y^4}^2 d\tau \nonumber\\
&&+ 12|a|\gamma \int_s^{\tau^*+s} \left\|\langle u\rangle \right\|_{L_y^\infty} \left\|u_{\notparallel} \right\|_{L^\infty} \left\|\partial_y^2 \langle u\rangle \right\|_{L^2} d\tau \nonumber \\
&& +24|a|\gamma \int_s^{\tau^*+s} \left\| \langle u\rangle \right\|_{L_y^\infty} \left\|\partial_y \langle u\rangle \right\|_{L_y^4} \left\|\nabla u_{\notparallel} \right\|_{L^4} d\tau  \nonumber \\
&\le& H_{2, 2, 1}+H_{2, 2, 2}+H_{2, 2, 3}, 
\end{eqnarray}
where
$$
H_{2, 2, 1}:=C\gamma \int_s^{\tau^*+s} \left\| \partial_y^2 \langle u \rangle \right\|_{L_y^2}^{\frac{1}{2}} \left\|\langle u\rangle\right\|_{L_y^2}^{\frac{3}{2}} \left\|\Delta u_{\notparallel} \right\|_{L^2} d\tau, 
$$
$$
H_{2, 2, 2}:=C\gamma \int_s^{\tau^*+s} \left\|\Delta u_{\notparallel} \right\|_{L^2}^{\frac{1}{2}} \left\|u_{\notparallel} \right\|_{L^2}^{\frac{1}{2}} \left\|\partial_y^2 \langle u \rangle \right\|_{L_y^2}^{\frac{5}{4}} \left\|\langle u \rangle \right\|_{L_y^2}^{\frac{3}{4}} d\tau, 
$$
and
$$
H_{2, 2, 3}:=C\gamma \int_s^{\tau^*+s} \left\|\partial_y^2 \langle u \rangle \right\|_{L_y^2}^{\frac{7}{8}} \left\|\langle u \rangle \right\|_{L_y^2}^{\frac{9}{8}} \left\|\Delta u_{\notparallel} \right\|_{L^2}^{\frac{3}{4}} \left\|u_{\notparallel} \right\|_{L^2}^{\frac{1}{4}} d\tau. 
$$
Note that here we have used the estimate 
\begin{eqnarray*}
&& 12|a| \gamma \int_s^{\tau^*+s} \left\|u_{\notparallel} \right\|_{L^\infty} \left\|\partial_y \langle u \rangle \right\|_{L_y^4}^2 d\tau \\
&& \quad \quad \quad  \quad \quad \quad +12 |a| \gamma \int_s^{\tau^*+s} \left\|\langle u\rangle \right\|_{L_y^\infty} \left\|u_{\notparallel} \right\|_{L^\infty} \left\|\partial_y^2 \langle u\rangle \right\|_{L_y^2} d\tau  \le H_{2, 2, 2}. 
\end{eqnarray*}

\medskip

\underline{Estimate of $H_{2, 2, 1}$.} By the bootstrap assumptions and Proposition \ref{avgprop}, we see that
\begin{eqnarray*}
H_{2, 2, 1}%
&\le& C\gamma \int_s^{\tau^*+s} \left\|\partial_y^2 \langle u \rangle \right\|_{L_y^2}^{\frac{1}{2}} \left\|\Delta u_{\notparallel} \right\|_{L^2} d\tau \\
&\le& C \gamma \left(\int_s^{\tau^*+s} d\tau \right)^{\frac{1}{4}}  \left(\int_s^{\tau^*+s} \left\|\partial_y^2 \langle u\rangle \right\|_{L_y^2}^2 d\tau \right)^{\frac{1}{4}} \left(\int_s^{\tau^*+s} \|\Delta u_{\notparallel} \|_{L^2}^2 d\tau \right)^{\frac{1}{2}}  \\
&\le& C \cdot \left(\frac{\gamma}{\lambda_{\gamma}} \right)^{\frac{1}{4}} \left\|u_{\notparallel}(s) \right\|_{L^2}. 
\end{eqnarray*}

\medskip

\underline{Estimate of $H_{2, 2, 2}$.} Using the bootstrap assumptions and Proposition \ref{avgprop} again, we have
\begin{eqnarray*}
H_{2, 2, 2}%
&\le& C\gamma \left\|u_{\notparallel}(s) \right\|_{L^2}^{\frac{1}{2}} \int_s^{\tau^*+s} \left\|u_{\notparallel} \right\|_{L^2}^{\frac{1}{2}} \left\|\partial_y^2 \langle u\rangle \right\|_{L_y^2}^{\frac{5}{4}} d\tau \\
&\le& C \gamma \left\|u_{\notparallel}(s) \right\|_{L^2}^{\frac{1}{2}} \left(\int_s^{\tau^*+s} d\tau \right)^{\frac{1}{8}} \left( \int_s^{\tau^*+s} \left\|\Delta u_{\notparallel} \right\|_{L^2}^2 d\tau \right)^{\frac{1}{4}} \\
&& \quad \quad \quad \quad \quad  \cdot \left(\int_s^{\tau^*+s} \left\|\partial_y^2 \langle u \rangle \right\|_{L^2}^2 d\tau \right)^{\frac{5}{8}} \\
&\le& C \cdot \left(\frac{\gamma}{\lambda_{\gamma}} \right)^{\frac{1}{8}} \left\|u_{\notparallel}(s)\right\|_{L^2}.
\end{eqnarray*}

\underline{Estimate of $H_{2, 2, 3}$.} Similarly, we have
\begin{eqnarray*}
H_{2, 2, 3}%
&\le& C\gamma  \|u_{\notparallel}(s)\|_{L^2}^{\frac{1}{4}} \int_s^{\tau^*+s} \left\|\partial_y^2 \langle u \rangle \right\|_{L_y^2}^{\frac{7}{8}} \left\|\Delta u_{\notparallel} \right\|_{L^2}^{\frac{3}{4}} d\tau \\
&\le& C \left\|u_{\notparallel}(s) \right\|_{L^2}^{\frac{1}{4}} \left(\int_s^{\tau^*+s} d\tau \right)^{\frac{3}{16}} \left( \int_s^{\tau^*+s} \left\|\partial_y^2 \langle u \rangle \right\|_{L_y^2}^2 d\tau \right)^{\frac{7}{16}} \\
&& \quad \quad \quad \quad \quad \cdot \left(\int_s^{\tau^*+s} \left\|\Delta u_{\notparallel} \right\|_{L^2}^2 d\tau \right)^{\frac{3}{8}} \\
&\le& C \cdot \left(\frac{\gamma}{\lambda_{\gamma}} \right)^{\frac{3}{16}} \left\|u_{\notparallel}(s) \right\|_{L^2}.
\end{eqnarray*}

To this end, combining all the estimates of $H_{2, 2, 1}, H_{2, 2, 2}$ and $H_{2, 2, 3}$ with \eqref{20210909eq20}, we conclude that
\begin{equation} \label{20210909eq22}
H_{2, 2} \le C \cdot \left[ \left(\frac{\gamma}{\lambda_{\gamma}} \right)^{\frac{1}{4}}+  \left(\frac{\gamma}{\lambda_{\gamma}} \right)^{\frac{1}{8}}+ \left(\frac{\gamma}{\lambda_{\gamma}} \right)^{\frac{3}{16}} \right] \cdot \|u_{\notparallel}(s)\|_{L^2}. 
\end{equation} 

\textbf{Estimate of $H_{2, 3}$.} First of all, we note that
$$
\Delta \left( \langle u\rangle u_{\notparallel}^2 \right)= 2\langle u\rangle \left|\nabla u_{\notparallel} \right|^2+2 \langle u \rangle u_{\notparallel} \Delta u_{\notparallel}+u_{\notparallel}^2 \partial_y^2 \langle u \rangle+4\partial_y \langle u \rangle u_{\notparallel} \partial_y u_{\notparallel}. 
$$
This gives
\begin{eqnarray*}
H_{2, 3}%
&\le& 12|a|\gamma \int_s^{\tau^*+s} \left\| \langle u \rangle \left|\nabla u_{\notparallel} \right|^2 \right\|_{L^2} d\tau+12|a| \gamma \int_s^{\tau^*+s} \left\|\langle u \rangle u_{\notparallel} \Delta u_{\notparallel} \right\|_{L^2} d\tau \\
&& +6|a|\gamma \int_s^{\tau*+s} \left\|u_{\notparallel} \partial_y^2 \langle u \rangle \right\|_{L^2}d\tau+24|a| \gamma \int_s^{\tau^*+s} \left\|\partial_y \langle u \rangle u_{\notparallel} \partial_y u_{\notparallel} \right\|_{L^2} d\tau. 
\end{eqnarray*}
Using the Gagliardo–Nirenberg's inequalities \eqref{GN1-1}, \eqref{GN1-2}, \eqref{GN2-1} and \eqref{GN2-2}, we can further bound the last term above by 
\begin{eqnarray} \label{20210909eq23}
&& 12 |a| \gamma \int_s^{\tau^*+s} \left\|\langle u\rangle \right\|_{L_y^\infty} \left\|\nabla u_{\notparallel} \right\|_{L^4}^2 d\tau \nonumber +6|a|\gamma \int_s^{\tau^*+s} \left\|u_{\notparallel}\right\|_{L^\infty}^2 \left\|\partial_y^2 \langle u \rangle \right\|_{L_y^2} d\tau \nonumber \\
&& \quad \quad +12|a|\gamma \int_s^{\tau^*+s} \left\|\langle u\rangle \right\|_{L_y^\infty} \left\|u_{\notparallel} \right\|_{L^\infty} \left\|\Delta u_{\notparallel} \right\|_{L^2} d\tau \nonumber \\ 
&& \quad \quad +24|a| \gamma \int_s^{\tau^*+s} \left\|u_{\notparallel} \right\|_{L^\infty} \left\|\partial_y \langle u\rangle \right\|_{L_y^4} \left\|\partial_y u_{\notparallel} \right\|_{L^4} d\tau \nonumber \\
&& \le  H_{2, 3, 1}+H_{2, 3, 2}+H_{2, 3, 3},
\end{eqnarray} 
where 
$$
H_{2, 3, 1}:=C\gamma \int_s^{\tau^*+s} \left\|\partial_y^2 \langle u\rangle \right\|_{L_y^2}^{\frac{1}{4}} \left\| \langle u \rangle \right\|_{L_y^2}^{\frac{3}{4}} \left\|\Delta u_{\notparallel} \right\|_{L^2}^{\frac{3}{2}} \left\|u_{\notparallel} \right\|_{L^2}^{\frac{1}{2}} d\tau,
$$
$$
H_{2, 3, 2}:={\bf L'} |a| \gamma \int_s^{\tau^*+s} \left\|\Delta u_{\notparallel} \right\|_{L^2} \left\|u_{\notparallel} \right\|_{L^2} \left\|\partial_y^2
 \langle u\rangle \right\|_{L^2} d\tau 
 $$
and
$$
H_{2, 3, 3}:=C\gamma \int_s^{\tau^*+s} \left\|\Delta u_{\notparallel} \right\|_{L^2}^{\frac{5}{4}} \left\|u_{\notparallel} \right\|_{L^2}^{\frac{3}{4}} \left\|\partial_y^2 \langle u\rangle \right\|_{L_y^2}^{\frac{5}{8}} \left\|\langle u\rangle \right\|_{L_y^2}^{\frac{3}{8}} d\tau.
$$
Here, ${\bf L'}>0$ is again some dimensional constant, and in the last estimate of \eqref{20210909eq23}, we have used the estimate 
\begin{eqnarray*}
&& 12|a| \gamma \int_s^{\tau^*+s} \left\|\langle u\rangle \right\|_{L_y^\infty} \left\|\nabla u_{\notparallel} \right\|_{L^4}^2 d\tau \\
&& \quad \quad \quad \quad +12 |a|\gamma \int_s^{\tau^*+s} \left\|\langle u\rangle \right\|_{L_y^\infty} \left\|u_{\notparallel} \right\|_{L^\infty} \left\|\Delta u_{\notparallel} \right\|_{L^2} d\tau \le H_{2, 3, 1}. 
\end{eqnarray*}

\underline{Estimate of $H_{2, 3, 1}$.} By the bootstrap assumptions and Proposition \ref{avgprop}, we have
\begin{eqnarray*}
H_{2, 3, 1}%
&\le& C\gamma \left\|u_{\notparallel}(s) \right\|_{L^2}^{\frac{1}{2}} \int_s^{\tau^*+s} \left\|\partial_y^2 \langle u \rangle \right\|_{L_y^2}^{\frac{1}{4}} \left\|\Delta u_{\notparallel} \right\|_{L^2}^{\frac{3}{2}} d\tau \\
&\le& C\gamma \left\|u_{\notparallel}(s) \right\|_{L^2}^{\frac{1}{2}} \left(\int_s^{\tau^*+s} d\tau \right)^{\frac{1}{8}} \left(\int_s^{\tau^*+s} \left\|\partial_y^2 \langle u\rangle \right\|_{L_y^2}^2 d\tau \right)^{\frac{1}{8}} \\
&&  \quad \quad \quad \quad \cdot \left(\int_s^{\tau^*+s} \left\|\Delta u_{\notparallel} \right\|_{L^2}^2 \right)^{\frac{3}{4}} \\
&\le& C \cdot \left(\frac{\gamma}{\lambda_\gamma} \right)^{\frac{1}{8}} \left\|u_{\notparallel}(s) \right\|_{L^2}^2 \le  C\left\|u_{\notparallel}(0)\right\|_{L^2}\cdot \left(\frac{\gamma}{\lambda_\gamma} \right)^{\frac{1}{8}} \left\|u_{\notparallel}(s) \right\|_{L^2}. 
\end{eqnarray*}

\underline{Estimate of $H_{2, 3, 2}$.} Using the bootstrap assumptions and Proposition \ref{avgprop} again, we have
\begin{eqnarray*}
H_{2, 3, 2}%
&\le& {\bf L'} |a| \gamma \left\|u_{\notparallel}(s)\right\|_{L^2}  \int_s^{\tau^*+s} \left\|\Delta u_{\notparallel} \right\|_{L^2} \left\|\partial_y^2 \langle u \rangle \right\|_{L_y^2} d\tau \\
&\le& \frac{20{\bf L'} |a| \left\|u_{\notparallel}(s) \right\|_{L^2}}{\varepsilon} \cdot  \left(\varepsilon  \gamma \int_s^{\tau^*+s} \left\|\Delta u_{\notparallel} \right\|_{L^2}^2 d\tau \right)^{\frac{1}{2}} \\
&&  \quad \quad \quad \quad \quad \quad \quad \quad \quad \quad \cdot \left(\varepsilon  \gamma \int_s^{\tau^*+s} \left\|\partial_y^2 \langle u\rangle \right\|^2_{L^2} d\tau \right)^{\frac{1}{2}} \\
&\le& \frac{200 {\bf L'} |a| {\bf B_3}^{\frac{1}{2}} \left\|u_{\notparallel}(s) \right\|^2_{L^2}}{\varepsilon} \\
&\le& \frac{4000 {\bf L'} |a| {\bf B_3}^{\frac{1}{2}} \left\|u_{\notparallel}(0) \right\|_{L^2}}{\varepsilon} \cdot \|u_{\notparallel}(s) \|_{L^2} \le \frac{\left\|u_{\notparallel}(s) \right\|_{L^2}}{20}.
\end{eqnarray*}

\underline{Estimate of $H_{2, 3, 3}$.} Similarly, we have  
\begin{eqnarray*}
H_{2, 3, 3}%
&\le& C \gamma \left\|u_{\notparallel}(s) \right\|_{L^2}^{\frac{3}{4}} \int_s^{\tau^*+s} \left\|\Delta u_{\notparallel}\right\|_{L^2}^{\frac{5}{4}} \left\|\partial_y^2 \langle u \rangle \right\|_{L_y^2}^{\frac{5}{8}} d\tau \\ 
&\le& C \left\|u_{\notparallel}(s)\right\|_{L^2}^{\frac{3}{4}} \left(\int_s^{\tau^*+s} d\tau \right)^{\frac{1}{16}} \left(\int_s^{\tau^*+s} \left\|\Delta u_{\notparallel} \right\|_{L^2}^2 d\tau \right)^{\frac{5}{8}} \\
&&  \quad \quad \quad \quad \cdot \left(\int_s^{\tau^*+s} \left\|\partial_y^2 \langle u \rangle \right\|_{L_y^2}^2 \right)^{\frac{5}{16}} \\
&\le& C \cdot \left(\frac{\gamma}{\lambda_\gamma} \right)^{\frac{1}{16}} \left\|u_{\notparallel}(s) \right\|_{L^2}^2 \le  C\left\|u_{\notparallel}(0)\right\|_{L^2}\cdot \left(\frac{\gamma}{\lambda_\gamma} \right)^{\frac{1}{16}} \left\|u_{\notparallel}(s) \right\|_{L^2}. 
\end{eqnarray*}

Combining all the estimates of $H_{2, 3, 1}, H_{2, 3, 2}$ and $H_{2, 3, 3}$ together with \eqref{20210909eq23}, this gives
\begin{equation} \label{20210910eq01}
 H_{2, 3} \le \frac{\|u_{\notparallel}(s)\|_{L^2}}{20} + C\left\|u_{\notparallel}(0) \right\|_{L^2} \left[\left(\frac{\gamma}{\lambda_{\gamma}} \right)^{\frac{1}{8}} +\left(\frac{\gamma}{\lambda_{\gamma}} \right)^{\frac{1}{16}} \right] \left\|u_{\notparallel}(s) \right\|_{L^2}.
\end{equation} 

\medskip

Finally, by \eqref{20210909eq01}, \eqref{20210909eq21}, \eqref{20210909eq22} and \eqref{20210910eq01}, we get
\begin{eqnarray} \label{20210910eq02}
H_2%
&\le& \frac{\|u_{\notparallel}(s)\|_{L^2}}{10}+ C \left[ \left(\frac{\gamma}{\lambda_{\gamma}} \right)^{\frac{1}{4}}+  \left(\frac{\gamma}{\lambda_{\gamma}} \right)^{\frac{1}{8}}+ \left(\frac{\gamma}{\lambda_{\gamma}} \right)^{\frac{3}{16}} \right] \cdot \|u_{\notparallel}(s)\|_{L^2} \nonumber \\
&&\quad \quad  + C\left\|u_{\notparallel}(0) \right\|_{L^2} \left[\left(\frac{\gamma}{\lambda_{\gamma}} \right)^{\frac{1}{8}} +\left(\frac{\gamma}{\lambda_{\gamma}} \right)^{\frac{1}{16}} \right] \left\|u_{\notparallel}(s) \right\|_{L^2}.
\end{eqnarray}
The estimate of $H_2$ is complete. 

\medskip

\textit{Estimate of $H_3$.} The estimate of $H_3$ is similar to the one of $H_2$, and hence we only sketch the proof here. Note that
$$
\Delta \left(u_{\notparallel}^2 \right)=2\left|\nabla u_{\notparallel} \right|^2+2u_{\notparallel} \Delta u_{\notparallel}
$$
and
$$
\Delta \left( \langle u \rangle u_{\notparallel} \right)=\langle u \rangle \Delta u_{\notparallel}+\partial_y^2 \langle u \rangle u_{\notparallel}+2\partial_y \langle u \rangle \partial_y u_{\notparallel}.
$$
Therefore,
\begin{eqnarray} \label{20210910eq03}
H_3%
&\le& 2|b|\gamma \int_s^{\tau^*+s} \left\|\Delta \left(u_{\notparallel}^2 \right) \right\|_{L^2} d\tau+4|b|\gamma \int_s^{\tau^*+s} \left\|\Delta \left(\langle u \rangle u_{\notparallel} \right) \right\|_{L^2} d\tau \nonumber \\
&\le& 4|b|\gamma \int_s^{\tau^*+s} \left\|\nabla u_{\notparallel} \right\|_{L^4}^2 d\tau+4|b|\gamma \int_s^{\tau^*+s} \left\|u_{\notparallel} \Delta u_{\notparallel} \right\|_{L^2} d\tau \nonumber \\
&&+4|b|\gamma \int_s^{\tau^*+s} \left\|\langle u\rangle u_{\notparallel} \right\|_{L^2} d\tau+4|b|\gamma \int_s^{\tau^*+s} \left\|\partial_y^2 \langle u \rangle u_{\notparallel} \right\|_{L^2} d\tau \nonumber \\
&&+8|b|\gamma \int_s^{\tau^*+s} \left\|\partial_y \langle u \rangle \partial_y u_{\notparallel} \right\|_{L^2} d\tau \nonumber \\
&\le& 4|b|\gamma \int_s^{\tau^*+s} \left\|\nabla u_{\notparallel} \right\|_{L^4}^2 d\tau+4|b|\gamma \int_s^{\tau^*+s} \left\|u_{\notparallel} \right\|_{L^\infty} \left\| \Delta u_{\notparallel} \right\|_{L^2} d\tau \nonumber \\
&&+4|b|\gamma \int_s^{\tau^*+s} \left\|\langle u \rangle \right\|_{L_y^\infty} \left\| u_{\notparallel} \right\|_{L^2} d\tau+4|b|\gamma \int_s^{\tau^*+s} \left\| u_{\notparallel} \right\|_{L^\infty} \left\|\partial_y^2 \langle u \rangle \right\|_{L_y^2} d\tau \nonumber \\
&&+8|b|\gamma \int_s^{\tau^*+s} \left\|\partial_y \langle u \rangle \right\|_{L_y^4} \left\|\partial_y u_{\notparallel} \right\|_{L^4} d\tau.
\end{eqnarray}
We bound the right hand side of \eqref{20210910eq03} again by using the bootstrap assumptions and Proposition \ref{avgprop}, which gives
\begin{equation} \label{20210910eq04}
H_3 \le C \left[\left(\frac{\gamma}{\lambda_{\gamma}} \right)^{\frac{3}{8}}+\left(\frac{\gamma}{\lambda_{\gamma}} \right)^{\frac{1}{4}}+\left(\frac{\gamma}{\lambda_{\gamma}} \right)^{\frac{5}{16}} \right] \left\|u_{\notparallel}(s) \right\|_{L^2}.
\end{equation}
The estimate of $H_4$ is complete. 

\medskip

Finally, we combine all the estimates for $H_1, H_2, H_3$ and $H_4$ together with \eqref{20210908eq10}, and this gives
\begin{eqnarray} \label{20210910eq11}
\left\|u_{\notparallel}(\tau^*+s) \right\|_{L^2}%
&\le& \left(\frac{10}{e^4}+\frac{1}{10} \right) \left\|u_{\notparallel}(s) \right\|_{L^2} \nonumber \\
&&+C\left(1+\left\|u_{\notparallel}(0) \right\|_{L^2} \right) \cdot \calF \left(\frac{\gamma}{\lambda_{\gamma}} \right) \left\|u_{\notparallel}(s) \right\|_{L^2},
\end{eqnarray}
where $\calF$ is a function on $\R_+$ satisfying $\calF(\alpha) \to 0$ as $\alpha \to 0$ (the explicit formula for $\calF$ can be derived directly from the estimates of $H_2, H_3$ and $H_4$). Note that $\frac{10}{e^4}+\frac{1}{10}<\frac{1}{e}$. Therefore, the desired estimate \eqref{20210905eq02} then follows if we pick a $\gamma_3>0$ sufficiently small, which only depends on $\varepsilon, a, b, B, \left\|\langle u \rangle(0) \right\|_{L^2}$ and $\left\| u_{\notparallel}(0) \right\|_{L^2}$, such that for any $0<\gamma \le \gamma_3$, one has
$$
C\left(1+\left\|u_{\notparallel}(0) \right\|_{L^2} \right) \cdot \calF \left(\frac{\gamma}{\lambda_{\gamma}} \right)  \le \frac{1}{e}-\frac{10}{e^4}-\frac{1}{10}.
$$
The proof is complete. 
\end{proof}

We are now ready to improve the first estimate in the bootstrap assumptions.

\begin{prop}
Assume the bootstrap assumptions, \eqref{20211017eq04} , \eqref{20211018eq01} and \eqref{20211026eq100}. Then there exists a $\gamma^*>0$, which only depends on $\varepsilon, a, b, \left\|\langle u \rangle(0) \right\|_{L^2}, \left\|u_{\notparallel}(0) \right\|_{L^2}$ and any dimensional constants, such that for any $0<\gamma \le \gamma^*$ and $0 \le s \le t \le t_0$, 
$$
\left\|u_{\notparallel}(t) \right\|_{L^2} \le 15 e^{\frac{-  \lambda_{\gamma}(t-s)}{4}} \left\|u_{\notparallel}(s) \right\|_{L^2}.
$$
\end{prop}

\begin{proof}
Let $0<\gamma^*<\min \left\{\gamma_0, \gamma_1, \gamma_2, \gamma_3 \right\}$, so that all the restrictions in Proposition \ref{avgprop}, Proposition \ref{bootestprop1}, Proposition \ref{bootestprop2} and Proposition \ref{bootestprop3} are satisfied. If $t_0<\tau^*$, then we have $t-s \le \tau^*=\frac{4}{\lambda_\gamma}$, and hence by Proposition \ref{bootestprop2},
$$
\left\|u_{\notparallel}(t) \right\|_{L^2} \le \frac{3}{2} \left\|u_{\notparallel}(s) \right\|_{L^2} \le \frac{15}{e} \left\|u_{\notparallel}(s)\right\|_{L^2} \le 15 e^{-\frac{\lambda_{\gamma}(t-s)}{4}} \left\|u_{\notparallel}(s) \right\|_{L^2}.
$$
If $t_0 \ge \tau^*$, then for any $0 \le s \le t \le t_0$, we can find some $n \in \Z_+$, such that $t \in [n \tau^*+s, (n+1)\tau^*+s)$. Then by Proposition \ref{bootestprop2} and Proposition \ref{bootestprop3}, we have
\begin{eqnarray*}
\left\|u_{\notparallel}(t) \right\|_{L^2}%
&\le& \frac{3}{2} \cdot \left\|u_{\notparallel}(n\tau^*+s) \right\|_{L^2} \le \frac{3}{2e^n} \cdot \left\|u_{\notparallel}(s) \right\|_{L^2} \\
&\le& \frac{3}{2} \cdot e^{\left(1-\frac{t-s}{\tau^*} \right)} \cdot \left\|u_{\notparallel}(s) \right\|_{L^2} \\
&=& \frac{3}{2} \cdot e^{-\frac{t-s}{\tau^*}} \cdot \left\|u_{\notparallel}(s) \right\|_{L^2} \\
&\le& 15 e^{-\frac{\lambda_\gamma(t-s)}{4}} \cdot \left\|u_{\notparallel}(s) \right\|_{L^2}.
\end{eqnarray*}
This concludes the proof of the proposition.
\end{proof}

\begin{proof} [Proof of Theorem \ref{mainthm}]
We first show that under the assumption of Theorem \ref{mainthm}, the solution to the problem \eqref{maineq02} exists globally. Note that it suffices to show that the maximal time $t_0=\infty$ in the bootstrap assumptions, while the regularity of the solution $u$ follows clearly from the bootstrap assumptions and Proposition \ref{bootestprop2}. Assume $t_0<\infty$. Then on the time interval $[0, t_0]$, we always have for $0 \le s \le t \le t_0$, 
\begin{enumerate}
    \item [(1).] $\|u_{\notparallel}(t)\|_{L^2} \le 15 e^{-\frac{ \lambda_{\gamma}(t-s)}{4}} \left\|u_{\notparallel}(s) \right\|_{L^2}$;
    \item [(2).] $\gamma \int_s^t \left\|\Delta u_{\notparallel}(\tau) \right\|_{L^2}^2 d\tau \le 5 \left\|u_{\notparallel}(s) \right\|_{L^2}^2$. 
\end{enumerate}
These estimates together with the continuity at $t_0$, we are able to find a $\epsilon>0$ sufficiently small, such that the bootstrap assumptions also hold on $[0, t_0+\epsilon]$, which contradicts to the maximality of $t_0$. The proof of the first part is complete.

Next we show that $\|u(t)\|_{L^2}$ decays exponentially. By the bootstrap assumptions and the fact that $t_0=\infty$, it suffices to show that $\|\langle u \rangle (t) \|_{L^2}$ decay exponentially, which is a consequence of a refined version of \eqref{20210829eq61}. More precisely, by \eqref{20210829eq609}, we have 
\begin{eqnarray*}
&& \frac{1}{2} \frac{d}{dt} \|\langle u \rangle\|_{L_y^2}^2+\frac{\varepsilon \gamma \lambda_1^2}{4} \|\langle u \rangle \|_{L_y^2}^2 \le  B\gamma \big(a^2\|\Delta u_{\notparallel}\|_{L^2}^2 \|u_{\notparallel}\|_{L^2}^4 \nonumber \\
&& \quad \quad \quad \quad \quad \quad \quad  \quad + \|\Delta u_{\notparallel}\|_{L^2} \|u_{\notparallel}\|_{L^2}^3+\|\Delta u_{\notparallel}\|_{L^2}^2 \|u_{\notparallel}\|_{L^2}^2 \|\langle u \rangle\|_{L_y^2}^2 \big),
\end{eqnarray*}
where we recall that $\lambda_1>0$ is the smallest positive eigenvalue of $-\Delta$ on $\T^2$. This gives
the following ODE:
\begin{eqnarray} \label{20211026eq200}
&& \frac{1}{2} \frac{d}{dt} \|\langle u \rangle\|_{L^2}^2+\left[\frac{\varepsilon \gamma \lambda_1^2}{4}-B\gamma \|\Delta u_{\notparallel}\|_{L^2}^2 \|u_{\notparallel}\|_{L^2}^2 \right] \|\langle u \rangle\|_{L_y^2}^2 \nonumber \\
&& \quad \quad \quad \quad \le  B\gamma \left(a^2\|\Delta u_{\notparallel}\|_{L^2}^2 \|u_{\notparallel}\|_{L^2}^4+ \|\Delta u_{\notparallel}\|_{L^2} \|u_{\notparallel}\|_{L^2}^3 \right). 
\end{eqnarray}
The proof then follows closely to the argument in Proposition \ref{avgprop}. Let
$$
\widetilde{\rho}(t):=\exp \left(\frac{\varepsilon \gamma \lambda_1^2 t}{4}-B\gamma \int_0^t \left\|\Delta u_{\notparallel}\right\|_{L^2}^2  \left\|u_{\notparallel} \right\|_{L^2}^2 d\tau \right)
$$
be the integrating factor and therefore
\begin{eqnarray} \label{20211027eq01}
&& \left\|\langle u \rangle(t) \right\|_{L_y^2}^2 \le \frac{\|\langle u \rangle(0) \|_{L_y^2}^2}{\widetilde{\rho}(t)} +\frac{Ba^2\gamma}{\widetilde{\rho}(t)} \int_0^t \widetilde{\rho}(\tau) \|\Delta u_{\notparallel}(t)\|_{L^2}^2 \|u_{\notparallel}(\tau)\|_{L^2}^4 d\tau \nonumber \\
&& \quad \quad \quad \quad \quad \quad \quad +\frac{B\gamma}{\widetilde{\rho}(t)} \int_0^t \widetilde{\rho}(\tau)  \|\Delta u_{\notparallel}(\tau)\|_{L^2} \|u_{\notparallel}(\tau)\|_{L^2}^3 d\tau.
\end{eqnarray}
Using \eqref{20210829eq63}, we have 
\begin{equation} \label{20211027eq02}
\frac{1}{\widetilde{\rho}(t)} \le \exp\left(-\frac{\varepsilon \gamma \lambda_1^2 t}{4}+B \left\|u_{\notparallel}(0) \right\|_{L^2} \right). 
\end{equation}
While for the two integrals in \eqref{20211027eq01}, by the bootstrap assumptions, we have 
\begin{eqnarray} \label{20211027eq03} 
&& \gamma \int_0^t \widetilde{\rho}(\tau) \left\|\Delta u_{\notparallel}(t) \right\|_{L^2}^2 \|u_{\notparallel}(\tau)\|_{L^2}^4 d\tau \nonumber \\
&& \quad \quad \quad \le  B\gamma \left\|u_{\notparallel}(0)\right\|_{L^2}^4 \int_0^t \exp \left(\frac{\varepsilon \gamma \lambda_1^2 t}{4}-\lambda_\gamma t \right) \left\|\Delta u_{\notparallel}(\tau) \right\|_{L^2}^2 d\tau 
\end{eqnarray}
and
\begin{eqnarray} \label{20211027eq04}
&& \gamma \int_0^t \widetilde{\rho}(\tau) \left\|\Delta u_{\notparallel}(\tau) \right\|_{L^2} \|u_{\notparallel}(\tau)\|_{L^2}^3 d\tau \nonumber \\
&& \quad \quad \quad B\gamma \|u_{\notparallel}(0)\|_{L^2}^3 \int_0^t \exp \left(\frac{\varepsilon \gamma \lambda_1^2 t}{4}-\frac{3\lambda_\gamma t}{4} \right) \left\|u_{\notparallel}(\tau) \right\|_{L^2}^2 d\tau.
\end{eqnarray}
Recall that $\lambda_\gamma \simeq \gamma^{\frac{2}{2+m}}$, therefore, for $\gamma$ sufficiently small, we have
$$
 \frac{\varepsilon \gamma \lambda_1^2}{4}-\frac{3\lambda_\gamma}{4} \le 0.  
$$
Therefore, \eqref{20211027eq03} and \eqref{20211027eq04} together with  \eqref{20210829eq64} and \eqref{20210829eq65}, respectively, give
\begin{equation} \label{20211027eq05}
\gamma \int_0^t \widetilde{\rho}(\tau) \left\|\Delta u_{\notparallel}(\tau) \right\|^2_{L^2} \|u_{\notparallel}(\tau)\|_{L^2}^4 d\tau \le B\|u_{\notparallel}(0)\|_{L^2}^6
\end{equation}
and 
\begin{equation} \label{20211027eq06}  
\gamma \int_0^t \widetilde{\rho}(\tau) \left\|\Delta u_{\notparallel}(\tau) \right\|_{L^2} \|u_{\notparallel}(\tau)\|_{L^2}^3 d\tau \le B\|u_{\notparallel}(0)\|_{L^2}^4
\end{equation}
Finally, combining \eqref{20211027eq01}, \eqref{20211027eq02}, \eqref{20211027eq05} and \eqref{20211027eq06}, we see that for any $t \ge 0$, 
\begin{eqnarray*}
\left\| \langle u \rangle(t) \right\|_{L_y^2}^2%
&\le&  B \exp\left(-\frac{\varepsilon \gamma \lambda_1^2 t}{4}+B\|u_{\notparallel}(0)\|_{L^2} \right) \\
&& \quad \quad \quad \cdot  \left[\left\|u_{\notparallel}(0)\right\|_{L^6}+\left\|\|u_{\notparallel}(0)\right\|_{L^4}+\left\| \langle u\rangle (0) \right\|^2_{L^2} \right].
\end{eqnarray*}
The proof is complete. 
\end{proof}

\end{document}